\documentclass[12pt,letterpaper]{article}

\usepackage{amsmath,amssymb,amsthm,mathtools}
\usepackage[margin=0.8in,a4paper]{geometry}    
\usepackage{cite} 
\usepackage{caption}
\usepackage{float}
\usepackage{tikz}
\usepackage{authblk}
\usepackage[colorlinks=true,linkcolor=blue,citecolor=blue,urlcolor=blue]{hyperref}
\usepackage{parskip}
\usepackage{enumitem}
 \usepackage{booktabs}
\usepackage{csquotes}
 \usepackage{lmodern}
\usepackage{orcidlink}
\usepackage{xcolor}
\usepackage{dsfont}
\usepackage{amscd}

\usetikzlibrary{decorations.pathreplacing,calc}

\theoremstyle{plain}
\newtheorem{theorem}{Theorem}
\newtheorem{corollary}{Corollary}
\newtheorem{lemma}{Lemma}
\newtheorem{proposition}{Proposition}

\theoremstyle{definition}
\newtheorem{definition}{Definition}

\newtheorem{problem}{Problem}

\theoremstyle{remark}
\newtheorem{remark}{Remark}

\DeclareRobustCommand{\seqnum}[1]{%
  \ifmmode
    \text{\href{https://oeis.org/#1}{\textcolor{blue}{{\normalfont\ttfamily #1}}}}%
  \else
    \href{https://oeis.org/#1}{\textcolor{blue}{{\normalfont\ttfamily #1}}}%
  \fi
}

\title{\textbf{Rotating-Memory Fibonacci Numbers \\and Periodic Tilings\thanks{\textcopyright W. Abdelaidoum, E.-M. Mehiri, H. Belbachir 2026}}}
 
\author[1]{Walid Abdelaidoum\,\orcidlink{0000-0001-6510-3236}%
\thanks{E-mail: \url{walidsd308@gmail.com}}}

\author[2]{El-Mehdi Mehiri\,\orcidlink{0000-0002-7164-3658}%
\thanks{Corresponding author. E-mails:
\url{elmehdi.mehiri@emse.fr} or
\url{mehiri314@gmail.com}}}

\author[1]{Hac\`ene Belbachir\,\orcidlink{0000-0001-8540-3033}%
\thanks{E-mail: \url{hbelbachir@usthb.dz}}}

\affil[1]{University of Science and Technology Houari Boumediene,
Faculty of Mathematics, RECITS Laboratory, P.O. Box 32, El-Alia,
16111 Bab-Ezzouar, Algiers, Algeria;}

\affil[2]{Mines Saint-\'Etienne, CMP,
Department of Manufacturing Sciences and Logistics,
F-13120 Gardanne, France}

\date{\today}

\begin{document}
\maketitle

\begin{abstract}
\noindent We introduce and study the \emph{rotating-memory Fibonacci numbers}, a periodic
variable-order analogue of the Fibonacci sequence in which the number of
preceding terms used in the recurrence changes cyclically with the index.
Despite this varying memory, the resulting sequences exhibit a remarkably
rigid structure. We derive closed forms, rational generating functions,
arithmetic properties, and exact growth behavior, and show that the
sequence decomposes naturally into geometric subsequences. We also develop
combinatorial interpretations in terms of periodically constrained
tilings, and restricted compositions, including
bijective explanations for the multiplicative structure of the sequence.
In addition, the first two nonclassical periods admit natural
geometry-driven realizations: the period-$2$ sequence arises from
monomer--dimer tilings of a triangular chain, while the period-$3$
sequence is related to tilings of a double hexagon strip by single and
double hexagons. These connections provide geometric interpretations of
the rotating recurrence in which the periodic behavior is induced by the
underlying structures themselves, and suggest a broader problem of
constructing analogous models for higher periods.
\end{abstract}

\medskip


\medskip


\section{Introduction}
\label{sec:introduction}

The Fibonacci sequence is one of the most extensively studied integer
sequences, and its simple second-order recurrence has inspired numerous
generalizations in number theory and combinatorics. One classical direction
is to increase the fixed order of the recurrence, replacing the sum of the
two preceding terms by the sum of a prescribed number of preceding terms.
Higher-order Fibonacci sequences have been studied for many decades; an
early systematic treatment, including a matrix formulation, was given by
Miles \cite{Miles1960}. These sequences and their variants continue to
provide a natural setting in which recurrence relations, generating
functions, matrix methods, and combinatorial enumeration interact.

Tilings provide one of the most useful combinatorial interpretations of
Fibonacci-type recurrences. The familiar interpretation of Fibonacci
numbers by tilings with monominoes and dominoes extends naturally to many
generalized recurrence sequences; see, for example, Benjamin and Quinn
\cite{BenjaminQuinn2003}. Benjamin and Heberle
\cite{BenjaminHeberle2014} developed tiling arguments for higher-order
$r$-Fibonacci numbers, while Benjamin, Derks, and Quinn
\cite{BenjaminDerksQuinn2011} showed more generally how linear
recurrences can be represented and studied through weighted tiling
models.  

Most familiar Fibonacci generalizations retain a recurrence of fixed
order. A different possibility is to allow the amount of memory itself to
vary with the index. This direction was considered by Emerson
\cite{Emerson2006}, who studied variable order Fibonacci-type recurrences
in which the $n$th term is obtained by summing a variable number of
preceding terms. His framework allows considerable freedom in the choice
of the order function and leads to a broad range of possible growth
behaviors. The present work focuses instead on a   structured
subclass in which the memory length varies periodically and consecutively.

More precisely, we consider a family in which, for a fixed period $k$, the
memory length rotates through $2,3,\ldots,k+1$ and then returns to $2$. We call the resulting sequences
\emph{Rotating-Memory Fibonacci numbers}, or RM-Fibonacci numbers. The
periodicity of the memory rule makes this family sufficiently flexible to
differ substantially from fixed-order Fibonacci sequences, while at the
same time imposing enough structure to permit an explicit analysis. One of the main themes of the paper is that a periodically varying
memory length can conceal a remarkably simple multiplicative structure.

\sloppy The combinatorial viewpoint plays an equally important role. We introduce
a one-dimensional tiling model in which the admissible length of a tile
depends periodically on the position at which the tile ends. This
endpoint-dependent rule translates the rotating recurrence directly into a
last tile decomposition. The same model can also be expressed in terms of
restricted compositions, where the admissibility of a part depends on the
partial sum at which it terminates. These interpretations
make it possible to understand some of the algebraic behavior of
RM-Fibonacci numbers bijectively rather than only through recurrence
manipulations.

% A further motivation comes from the connection between recurrence
% sequences and tilings on the hexagonal lattice. Jin and Dresden
% \cite{JinDresden2022} used hexagonal double strip tilings to obtain
% Tetranacci identities. More recently, Do\v{s}li\'c and Podrug
% \cite{DoslicPodrug2024} developed tiling models on honeycomb strips that
% lead to higher-order Fibonacci-type sequences, including Tribonacci,
% Padovan, and Narayana's cow numbers. Within the RM-Fibonacci family, the
% period-$3$ case leads, after an index shift, to  sequence \seqnum{A354541} in  \emph{The On-Line Encyclopedia of Integer Sequences} \cite{OEIS}, which counts tilings of a double hexagon strip by
% single and double hexagons. This coincidence provides two quite different
% combinatorial realizations of the same integer sequence and motivates the
% search for a direct connection between the corresponding tiling models.

 Beyond these general position-dependent models, the first two
nonclassical periods admit natural geometric realizations in which the
periodicity is induced by the underlying structure itself. For period
$2$, the RM-Fibonacci sequence is, up to the initial shift recorded in
Table~\ref{tab:RM-initial-values}, sequence \seqnum{A038754} in  \emph{The On-Line Encyclopedia of Integer Sequences} \cite{OEIS}. We show
that its terms enumerate monomer--dimer tilings of a triangular chain
when one boundary vertex is prescribed to be a monomer.  

A related phenomenon occurs for period $3$. Hexagonal tilings have
previously been used to realize several Fibonacci-type sequences. Jin
and Dresden \cite{JinDresden2022} used hexagonal double-strip tilings to
obtain Tetranacci identities, while Do\v{s}li\'c and Podrug
\cite{DoslicPodrug2024} developed tiling models on honeycomb strips
leading to Tribonacci, Padovan, Narayana's cow, and related sequences.
In our setting, the shifted period-$3$ RM-Fibonacci sequence coincides
with \seqnum{A354541}, which enumerates tilings of a double hexagon strip
by single and double hexagons. Together, the triangular-chain and
double-hexagon models suggest that rotating memory may admit natural
geometric realizations beyond the explicitly position-dependent tiling
construction.
 
The remainder of the paper is organized as follows. Section~\ref{sec:definition} introduces the RM-Fibonacci numbers and illustrates the construction for several small periods. Section~\ref{sec:collapse} develops the structural properties of the recurrence and establishes the period-collapse phenomenon. Section~\ref{sec:closed-forms} derives closed forms, arithmetic properties, generating functions, and growth results. Section~\ref{sec:tilings} develops the combinatorial interpretations in terms of endpoint-periodic tilings, and restricted compositions,  together with bijective explanations of the underlying multiplicative structure.  Section~\ref{sec:triangular-period2} gives a geometric realization of the period-$2$ sequence in
terms of monomer--dimer tilings of a triangular chain. Section~\ref{sec:hexagonal} focuses on the period-$3$ case and its connection with tilings of a double hexagon strip. Finally, Section~\ref{sec:conclusion} summarizes the main findings and discusses open problems and directions for further research.

 \section{Rotating-Memory Fibonacci Numbers}
\label{sec:definition}

We begin with the basic definition. Throughout the paper, if an index is
negative, the corresponding sequence term is understood to be zero.

\begin{definition}\label{def:RMF}
Let $k\geq1$. The \emph{Rotating-memory Fibonacci sequence of period $k$}
is the sequence $\bigl(R_n^{(k)}\bigr)_{n\geq0}$ defined by
\begin{equation}
    R_0^{(k)}=0,\qquad R_1^{(k)}=1,\qquad \forall n\geq2: R_n^{(k)}
=
\sum_{j=1}^{\,2+(n\bmod k)}
R_{n-j}^{(k)},\label{eq:RM-definition}
\end{equation} 
where $R_m^{(k)}=0$ for all $m<0$.   
\end{definition}

Thus the memory length is the periodic function $\ell_k(n)=2+(n\bmod k)$, whose values, according to the residue class of $n$, are $2,3,\ldots,k+1$.

For $k=1$, we recover the classical Fibonacci sequence, since
$n\bmod 1=0$ for every $n$. Hence
\[
R_n^{(1)}
=
R_{n-1}^{(1)}+R_{n-2}^{(1)},
\qquad n\geq2,
\]
with $R_0^{(1)}=0$ and  $R_1^{(1)}=1$.

For $k=2$, the memory length alternates between $2$ and $3$. Therefore
\[
R_n^{(2)}
=
\begin{cases}
R_{n-1}^{(2)}+R_{n-2}^{(2)},
& n\equiv0\pmod2,\\[1mm]
R_{n-1}^{(2)}+R_{n-2}^{(2)}+R_{n-3}^{(2)},
& n\equiv1\pmod2,
\end{cases}
\]
with $R_0^{(2)}=0$, $R_1^{(2)}=1$, and $R_m^{(2)}=0$ for $m<0$.

For $k=3$, the memory length rotates through $2$, $3$, and $4$. Hence
\[
R_n^{(3)}
=
\begin{cases}
R_{n-1}^{(3)}+R_{n-2}^{(3)},
& n\equiv0\pmod3,\\[1mm]
R_{n-1}^{(3)}+R_{n-2}^{(3)}+R_{n-3}^{(3)},
& n\equiv1\pmod3,\\[1mm]
R_{n-1}^{(3)}+R_{n-2}^{(3)}+R_{n-3}^{(3)}+R_{n-4}^{(3)},
& n\equiv2\pmod3,
\end{cases}
\]
with $R_0^{(3)}=0$, $R_1^{(3)}=1$,  and $R_m^{(3)}=0$ for $m<0$.

Initial values of the RM-Fibonacci sequences for the first few periods and their OEIS correspondences are given in Table~\ref{tab:RM-initial-values}.
\begin{table}[H]
\centering

\begin{tabular}{c|l|l}
\toprule
$k$ & $R_0,R_1,\ldots,R_{16}$& OEIS\\
\midrule
1 &
$0,1,1,2,3,5,8,13,21,34,55,89,144,233,377,610,987$& \seqnum{A000045}$(n)$ \\
2 &
$0,1,1,2,3,6,9,18,27,54,81,162,243,486,729,1458,2187$&\seqnum{A038754}$(n-2)$ \\
3 &
$0,1,1,2,4,8,12,24,48,72,144,288,432,864,1728,2592,5184$&\seqnum{A354541}$(n-1)$ \\
4 &
$0,1,1,2,3,6,12,24,36,72,144,288,432,864,1728,3456,5184$&- \\
5 &
$0,1,1,2,4,6,12,24,48,96,144,288,576,1152,2304,3456,6912$&- \\
\bottomrule
\end{tabular}
\caption{Initial values of the RM-Fibonacci sequences for periods
$k=1,\ldots,5$.}
\label{tab:RM-initial-values}
\end{table}
Although the rows look rather different near the origin, the structure
after one complete period is remarkably rigid.

We first determine the initial transient.

\begin{lemma}
\label{lem:initial}
For $k\geq2$, the initial terms of the RM-Fibonacci sequence satisfy
\[
R_n^{(k)}
=
\begin{cases}
1, & k=2,\ n=2,\\[1mm]
2^{n-2}, & k\geq3,\ 2\leq n\leq k-1,\\[1mm]
2, & k=3,\ n=k,\\[1mm]
3\cdot2^{k-4}, & k\geq4,\ n=k.
\end{cases}
\]
\end{lemma}
\begin{proof}
Let $2\leq n\leq k-1$. Since $n\bmod k=n$, the memory length at $n$ is
$n+2$. All terms with negative indices vanish, so
\[
R_n^{(k)}
=
\sum_{j=1}^{n}R_{n-j}^{(k)}
=
\sum_{i=0}^{n-1}R_i^{(k)}.
\]
Since $R_0^{(k)}=0$ and $R_1^{(k)}=1$, it follows successively that $R_2^{(k)}=1$, 
$R_3^{(k)}=2$, 
$R_4^{(k)}=4$, and an immediate induction gives $R_n^{(k)}=2^{n-2}$ for $2\leq n\leq k-1$.

At $n=k$ the cycle resets, so the recurrence has order $2$:
\[
R_k^{(k)}
=
R_{k-1}^{(k)}+R_{k-2}^{(k)}.
\]
For $k\geq4$, the formula just proved gives
\[
R_k^{(k)}
=
2^{k-3}+2^{k-4}
=
3\cdot2^{k-4}.
\]
The cases $k=2$ and $k=3$ follow directly from the initial values.
\end{proof}

The next section shows that after this short initial phase no complicated
high-order analysis is needed.

\section{Collapse of the Rotating Recurrence}
\label{sec:collapse}

The main structural fact is that consecutive terms are related by simple
multiplicative factors. We distinguish between a \emph{reset position},
meaning an index divisible by $k$, and all other positions.

\begin{lemma} 
\label{lem:doubling}
Let $k\geq2$, and let $n\geq3$ satisfy $n\not\equiv0\pmod{k}$. Then $ R_n^{(k)}=2R_{n-1}^{(k)}$. 
\end{lemma}

\begin{proof}
Write $r=n\bmod k$ with $1\leq r\leq k-1$. By Definition~\ref{def:RMF},
\[
R_n^{(k)}
=
\sum_{j=1}^{r+2}R_{n-j}^{(k)}.
\]
Since $(n-1)\bmod k=r-1$,  the recurrence at $n-1$ has memory length $r+1$, and therefore
\[
R_{n-1}^{(k)}
=
\sum_{j=1}^{r+1}R_{n-1-j}^{(k)}
=
\sum_{j=2}^{r+2}R_{n-j}^{(k)}.
\]
Consequently,
\[
R_n^{(k)}
=
R_{n-1}^{(k)}
+
\sum_{j=2}^{r+2}R_{n-j}^{(k)}
=
2R_{n-1}^{(k)}.\qedhere
\]
\end{proof}

Thus every time the memory increases by one, the sequence simply doubles.
Only the reset from memory $k+1$ back to memory $2$ behaves differently.

\begin{lemma} 
\label{lem:reset}
Let $k\geq2$, and let $n\geq4$ satisfy $n\equiv0\pmod{k}$. Then $R_n^{(k)}
=
\frac32 R_{n-1}^{(k)}$. 
\end{lemma}

\begin{proof}
At a reset position the memory length is $2$, and hence
\[
R_n^{(k)}
=
R_{n-1}^{(k)}+R_{n-2}^{(k)}.
\]
Since $n-1\equiv k-1\pmod{k}$, Lemma~\ref{lem:doubling} gives $R_{n-1}^{(k)}
=
2R_{n-2}^{(k)}$.  Therefore $R_{n-2}^{(k)}
=
\frac12R_{n-1}^{(k)}$,  and hence
\[
R_n^{(k)}
=
R_{n-1}^{(k)}
+
\frac12R_{n-1}^{(k)}
=
\frac32R_{n-1}^{(k)}.\qedhere
\]
\end{proof}

We can now state the main structural theorem, called the \emph{period-collapse theorem}.

\begin{theorem}
\label{thm:period-collapse}
Let $k\geq2$. Then, for every $n\geq k$,
\begin{equation}
\label{eq:period-collapse}
R_{n+k}^{(k)}
=
3\cdot2^{k-2}R_n^{(k)}.
\end{equation}
\end{theorem}

\begin{proof}
Consider the $k$ successive transitions $R_n^{(k)}
\longrightarrow
R_{n+1}^{(k)}
\longrightarrow\cdots\longrightarrow
R_{n+k}^{(k)}$. Among the indices $n+1,n+2,\ldots,n+k$ there is exactly one multiple of $k$. At that index,
Lemma~\ref{lem:reset} contributes the factor $3/2$.
Each of the other $k-1$ transitions occurs between resets, and
Lemma~\ref{lem:doubling} contributes a factor $2$.

Since $n\geq k$, all the indices involved are beyond the exceptional
initial positions. Therefore
\[
\frac{R_{n+k}^{(k)}}{R_n^{(k)}}
=
\frac32\,2^{k-1}
=
3\cdot2^{k-2}
\Longrightarrow   R_{n+k}^{(k)}  
=
3\cdot2^{k-2}R_n^{(k)}.\qedhere
\] 
\end{proof}

It is useful to introduce the period multiplier $c_k:=3\cdot2^{k-2}$. Theorem~\ref{thm:period-collapse} then reads simply
\begin{equation}
R_{n+k}^{(k)}=c_kR_n^{(k)}.\label{eq:constant-coefficient-recurrence}
\end{equation}
This result shows that the variable order recurrence
\eqref{eq:RM-definition}, whose largest memory length is $k+1$, eventually
reduces to the sparse constant coefficient recurrence
\[
R_{n+k}^{(k)}-c_kR_n^{(k)}=0.
\]
In particular, each residue class modulo $k$ evolves geometrically.

\begin{corollary} 
\label{cor:geometric-subsequences}
Let $k\geq2$, $0\leq r\leq k-1$, and $q\geq1$. Then
\begin{equation}
\label{eq:block-form}
R_{qk+r}^{(k)}
=
2^r R_k^{(k)}c_k^{q-1}.
\end{equation}
\end{corollary}

\begin{proof}
By Lemma~\ref{lem:doubling}, we have  $R_{k+r}^{(k)}
=
2^rR_k^{(k)}$ with  $0\leq r\leq k-1$. Repeated application of Theorem~\ref{thm:period-collapse} now gives
\[
R_{qk+r}^{(k)}
=
c_k^{q-1}R_{k+r}^{(k)}
=
2^rR_k^{(k)}c_k^{q-1}.\qedhere
\]
\end{proof}

Thus the entire sequence beyond its initial transient is obtained by
repeating the same block shape $1,2,4,\ldots,2^{k-1}$, while multiplying successive blocks by $c_k$.

For $k\geq4$, Lemma~\ref{lem:initial} gives an especially simple formula.

\begin{corollary}
\label{cor:explicit-qkr}
Let $k\geq4$, $q\geq1$, and $0\leq r\leq k-1$. Then
\begin{equation}\label{eq:explicit-qkr}
R_{qk+r}^{(k)}
=
3^q\,2^{(k-2)q+r-2}
\end{equation}
\end{corollary}

\begin{proof}
Using $R_k^{(k)}=3\cdot2^{k-4}$ and $c_k=3\cdot2^{k-2}$ in \eqref{eq:block-form}, we obtain
\[
R_{qk+r}^{(k)}
=
2^r
\left(3\cdot2^{k-4}\right)
\left(3\cdot2^{k-2}\right)^{q-1}.
\]
Collecting powers of $2$ and $3$ gives~\eqref{eq:explicit-qkr}. 
\end{proof}

Consequently, for every $k\ge4$, all terms after the initial transient
are $3$-smooth; that is, their only prime divisors are $2$ and $3$.

The period-collapse theorem also determines the consecutive term ratios exactly.

\begin{corollary} 
\label{cor:ratio-pattern}
Let $k\geq2$. For every $n\geq4$,
\[
\frac{R_n^{(k)}}{R_{n-1}^{(k)}}
=
\begin{cases}
\dfrac{3}{2}, & n\equiv0\pmod{k},\\[2mm]
2, & n\not\equiv0\pmod{k}.
\end{cases}
\]
Consequently, the sequence of consecutive term ratios is periodic with
period $k$ from $n=4$ onward.
\end{corollary}

In particular, the oscillatory growth is not merely asymptotically
periodic: the periodic modulation becomes exact after finitely many terms.
This observation will allow us to derive closed forms, generating
functions, and asymptotic results without recourse to periodic matrix
products.

Note that periodic dependence on the index has also been studied from a different
perspective.  Panario, Sahin, Wang, and Webb
\cite{PanarioSahinWangWebb2014} introduced general conditional
recurrences, in which the coefficients of a fixed-order linear recurrence
depend on the residue class of the index.  Such sequences satisfy a
constant coefficient recurrence and admit rational generating functions.
The RM-Fibonacci recurrence may be embedded in this framework by taking
the maximal order $k+1$ and allowing the coefficients to be periodic
zeros and ones.  What is special in the present setting is the highly
structured staircase pattern of these coefficients: it forces the general
constant coefficient reduction to collapse to the   sparse
relation~\eqref{eq:period-collapse}  which in turn yields exact consecutive term ratios, geometric
residue class subsequences, and the combinatorial cycle decomposition
developed below.

Another related direction concerns Fibonacci recurrences with periodic
coefficients.  Edson, Lewis, and Yayenie
\cite{EdsonLewisYayenie2011} introduced $k$-periodic Fibonacci sequences
in which the recurrence remains of order two while its coefficients vary
periodically.  The present construction is complementary: the
coefficients of the active terms remain equal to one, while the number
of active preceding terms itself changes periodically.

 \section{Closed Forms, Generating Functions, and Growth}
\label{sec:closed-forms}

The period-collapse theorem makes it possible to describe the
RM-Fibonacci sequence almost completely. In particular, the apparently
non-autonomous recurrence has a rational generating function and an exact
periodically modulated exponential form.

\subsection{Closed forms}

After the first period, \eqref{eq:block-form} already gives a complete closed form.  For the two smallest nonclassical periods, we obtain   simple
expressions.

\begin{proposition}
\label{prop:closed-small-k}
For $q\geq1$, we have
\begin{align*}
    R_{2q}^{(2)}&=3^{q-1},
\\
R_{2q+1}^{(2)}&=2\cdot3^{q-1},
\end{align*}
and
\begin{align*}
    R_{3q}^{(3)}&=2\cdot6^{q-1},
\\
R_{3q+1}^{(3)}&=4\cdot6^{q-1},
\\
R_{3q+2}^{(3)}&=8\cdot6^{q-1}.
\end{align*}
\end{proposition}

\begin{proof}
For $k=2$ we have $R_2^{(2)}=1$ and $c_2=3$,  whereas for $k=3$, $R_3^{(3)}=2$ and $c_3=6$. The result follows immediately from~\eqref{eq:block-form}. 
\end{proof}

For $k\geq4$, Formula~\eqref{eq:explicit-qkr} also determines the exact powers of $2$ and $3$
appearing in each term.

\begin{corollary}
\label{cor:valuations}
Let $k\geq4$, $q\geq1$, and $0\leq r\leq k-1$. Then
\begin{align*}
\nu_3\!\left(R_{qk+r}^{(k)}\right)&=q\\
\nu_2\!\left(R_{qk+r}^{(k)}\right)
&=
(k-2)q+r-2,
\end{align*}
where $\nu_p(m)$ denotes the $p$-adic valuation of $m$.
\end{corollary}

This gives a   transparent arithmetic description of the
sequence: the exponent of $3$ records the number of completed memory
cycles, while the exponent of $2$ records both the cycle number and the
position within the current cycle.

The sums of complete blocks are equally simple.

\begin{proposition}
\label{prop:block-sums}
For $q\geq1$, we have
\[
\sum_{r=0}^{k-1}R_{qk+r}^{(k)}
=
\begin{cases}
(2^k-1)R_k^{(k)}c_k^{q-1},
& k=2,3,\\[1mm]
(2^k-1)3^q2^{(k-2)q-2},
& k\geq4.
\end{cases}
\]
\end{proposition}

\begin{proof}
Using Corollary~\ref{cor:geometric-subsequences},
\[
\sum_{r=0}^{k-1}R_{qk+r}^{(k)}
=
R_k^{(k)}c_k^{q-1}
\sum_{r=0}^{k-1}2^r.
\]
Since
\[
\sum_{r=0}^{k-1}2^r=2^k-1,
\]
the first formula follows. The second follows from
Corollary~\ref{cor:explicit-qkr}.
\end{proof}

\subsection{Rational generating functions}

A periodic recurrence with variable order might initially suggest a more
complicated generating-function structure. In the present case, however,
Theorem~\ref{thm:period-collapse} shows that the ordinary generating
function is rational. Let $G_k(x)
=
\sum_{n\geq0}R_n^{(k)}x^n$.

\begin{theorem}
\label{thm:general-gf}
For every $k\geq2$, we have 
\[
G_k(x)
=
\frac{P_k(x)}
{1-c_kx^k},
\]
where 
\[
P_k(x)
=
\sum_{n=0}^{k-1}R_n^{(k)}x^n
+
\sum_{r=0}^{k-1}
\left(
R_{k+r}^{(k)}-c_kR_r^{(k)}
\right)x^{k+r}.
\]
\end{theorem}

\begin{proof}
Multiplying $G_k(x)$ by $1-c_kx^k$ gives
\[
(1-c_kx^k)G_k(x)
=
\sum_{n\geq0}R_n^{(k)}x^n
-
c_k
\sum_{n\geq0}R_n^{(k)}x^{n+k}.
\]
Hence the coefficient of $x^m$ is $R_m^{(k)}$ when $m<k$, and $R_m^{(k)}-c_kR_{m-k}^{(k)}$ when $m\geq k$.

By Theorem~\ref{thm:period-collapse}, we have
\[
R_m^{(k)}=c_kR_{m-k}^{(k)},
\]
for all $m\geq2k$, so all coefficients of degree at least $2k$ vanish. Therefore
$(1-c_kx^k)G_k(x)$ is the polynomial $P_k(x)$ displayed above.
\end{proof}

For $k=2$ and $k=3$, this yields the following expressions.

\begin{corollary}
\label{cor:gf-small}
We have  
\[
\begin{aligned}
    G_2(x)
&=
\frac{x+x^2-x^3}{1-3x^2},\\
G_3(x)
&=
\frac{x+x^2+2x^3-2x^4+2x^5}
{1-6x^3}.
\end{aligned}
\]
\end{corollary}

For $k\geq4$, almost all terms in the second sum defining $P_k(x)$
cancel, leading to a   compact formula.

\begin{theorem}
\label{thm:gf-explicit}
Let $k\geq4$. Then
\[
G_k(x)
=
\frac{
x
+
\displaystyle\sum_{j=2}^{k-1}2^{j-2}x^j
+
3\cdot2^{k-4}x^k
-
3\cdot2^{k-3}x^{k+1}
}
{
1-3\cdot2^{k-2}x^k
}.
\]
Equivalently,
\[
G_k(x)
=
\frac{
x
+
\displaystyle
x^2\frac{1-(2x)^{k-2}}{1-2x}
+
3\cdot2^{k-4}x^k(1-2x)
}
{
1-3\cdot2^{k-2}x^k
}.
\]
\end{theorem}

\begin{proof}
For $0\leq r\leq k-1$,
Corollary~\ref{cor:geometric-subsequences} gives $R_{k+r}^{(k)}
=
2^rR_k^{(k)}$.  For $k\geq4$, $R_k^{(k)}=3\cdot2^{k-4}$ and $c_k
=
3\cdot2^{k-2}
=
4R_k^{(k)}$.

For $r\geq2$, $R_r^{(k)}=2^{r-2}$,  so
\[
R_{k+r}^{(k)}-c_kR_r^{(k)}
=
2^rR_k^{(k)}
-
4R_k^{(k)}2^{r-2}
=
0.
\]
For $r=0$,
\[
R_k^{(k)}-c_kR_0^{(k)}
=
R_k^{(k)},
\]
while for $r=1$,
\[
R_{k+1}^{(k)}-c_kR_1^{(k)}
=
2R_k^{(k)}-4R_k^{(k)}
=
-2R_k^{(k)}.
\]
Substitution into Theorem~\ref{thm:general-gf} yields the result.
\end{proof}

Thus, despite being defined by a periodically changing recurrence order,
the RM-Fibonacci sequence belongs to the class of C-finite sequences
after a finite initial transient.

The constant coefficient recurrence~\eqref{eq:constant-coefficient-recurrence} has characteristic polynomial
\begin{equation}
\label{eq:characteristic}
\chi_k(z)=z^k-c_k
=
z^k-3\cdot2^{k-2}.
\end{equation}

Let $\rho_k
=
c_k^{1/k}
=
\left(3\cdot2^{k-2}\right)^{1/k}$. If $\zeta_k=e^{2\pi i/k}$,  then the roots of \eqref{eq:characteristic} are
\[
\rho_k,
\;
\rho_k\zeta_k,
\;
\rho_k\zeta_k^2,
\;
\ldots,
\;
\rho_k\zeta_k^{k-1}.
\]
Unlike the usual Fibonacci recurrence, there is no single characteristic
root having strictly greater modulus than all the others: every root has
modulus $\rho_k$. This is precisely what produces the periodic
modulation according to the residue class of $n$.

Indeed, writing $n=qk+r$ with $0\leq r\leq k-1$,   we obtain $R_n^{(k)}
=
C_{k,r}\rho_k^n$ for all $q\geq1$, where $C_{k,r}
=
\dfrac{2^rR_k^{(k)}}{c_k\rho_k^r}$. Thus the normalized sequence $\dfrac{R_n^{(k)}}{\rho_k^n}$ is eventually periodic with period $k$.

\subsection{Exponential growth}

The preceding formulas determine the exponential growth exactly.

\begin{theorem}
\label{thm:growth}
For every $k\geq2$,  we have 
\begin{equation}
\label{eq:rho}
\lim_{n\to\infty}
\left(R_n^{(k)}\right)^{1/n}
=\rho_k
=
2\left(\frac34\right)^{1/k}.
\end{equation}
\end{theorem}

\begin{proof}
Let $n=qk+r$ with $0\leq r\leq k-1$. Using Corollary~\ref{cor:geometric-subsequences},
\[
R_n^{(k)}
=
2^rR_k^{(k)}c_k^{q-1}.
\]
Therefore
\[
\frac{1}{n}\log R_n^{(k)}
=
\frac{
r\log2+\log R_k^{(k)}+(q-1)\log c_k
}
{qk+r}.
\]
As $n\to\infty$, necessarily $q\to\infty$, and hence
\[
\lim_{n\to\infty}
\frac{1}{n}\log R_n^{(k)}
=
\frac{1}{k}\log c_k.
\]
Exponentiating gives
\[
\lim_{n\to\infty}
\left(R_n^{(k)}\right)^{1/n}
=
c_k^{1/k}.
\]
Finally,
\[
c_k^{1/k}
=
\left(3\cdot2^{k-2}\right)^{1/k}
=
2\left(\frac34\right)^{1/k}.\qedhere
\]
\end{proof}

The growth rate has a simple dependence on the rotation period.

\begin{proposition}
\label{prop:rho-monotone}
\sloppy The sequence $(\rho_k)_{k\geq2}$ is strictly increasing and $\sqrt{3}
=
\rho_2
<
\rho_3
<
\rho_4
<
\cdots
<
2$. Moreover, $\lim_{k\to\infty}\rho_k=2$. 
\end{proposition}

\begin{proof}
From \eqref{eq:rho},
\[
\log\rho_k
=
\log2+\frac{1}{k}\log\frac34.
\]
Since $\log\frac34<0$,  the quantity $\frac{1}{k}\log\frac34$ is strictly increasing with $k$. Hence $\rho_k$ is strictly increasing.

Furthermore, $\left(\frac34\right)^{1/k}<1$,  so $\rho_k<2$ for every finite $k$, while $\lim_{k\to\infty}
\left(\frac34\right)^{1/k}=1$. Therefore $\lim_{k\to\infty}\rho_k=2$. 
\end{proof}

Consequently, allowing a longer rotation of the memory lengths increases
the average exponential growth rate. The limiting value $2$ is natural:
within each period, all but one transition are exact doublings, while the
single reset transition has the smaller factor $3/2$.

The global growth rate $\rho_k$ describes the geometric mean over a
period, but consecutive terms alternate between two exact local growth
factors. Indeed,
\[
\prod_{j=1}^{k}
\frac{R_{n+j}^{(k)}}{R_{n+j-1}^{(k)}}
=
\frac32\,2^{k-1}
=
c_k,
\]
and hence
\[
\rho_k
=
\left(
\frac32\,2^{k-1}
\right)^{1/k}.
\]

Thus the exponential growth rate is the geometric mean of the $k$
local multipliers occurring during one memory cycle, i.e., 
\[
\rho_k
=
\left(
\underbrace{2\cdot 2\cdots 2}_{k-1\text{ times}}
\cdot\frac32
\right)^{1/k}.
\]
This reconciles the periodic behavior of the consecutive quotients with
the existence of a single global exponential growth rate.

\section{Periodic Tilings and  Restricted Compositions}
\label{sec:tilings}

Before introducing the formal tiling model, it is useful to keep a simple
picture in mind. Imagine a worker constructing a long wall from left to
right while the available tile lengths change periodically. At successive
stages, the supplier offers tiles of lengths
\[
1,2;\qquad
1,2,3;\qquad
1,2,3,4;\qquad
\ldots;\qquad
1,2,\ldots,k+1,
\]

after which the same cycle begins again. Thus the set of admissible tile
lengths depends on the phase of the construction.

For example, when $k=3$, the available lengths cycle as
\[
\{1,2\},\qquad
\{1,2,3\},\qquad
\{1,2,3,4\},\qquad
\{1,2\},\qquad\ldots.
\]
Consequently, when the right endpoint of the wall reaches a given
position, the possible choices for the final tile depend periodically on
that position. A last-tile decomposition therefore produces a recurrence
whose order varies periodically. This is precisely the combinatorial
mechanism underlying the RM-Fibonacci numbers.

The structural simplicity of RM-Fibonacci numbers thus has a natural
tiling interpretation. As with the classical Fibonacci and higher-order
Fibonacci sequences, the recurrence can be represented by tilings of a
one-dimensional board. The essential difference is that the admissible
tile lengths now depend periodically on the position at which a tile
ends.

Since the empty tiling is naturally counted by $1$, it is convenient to
shift the RM-Fibonacci sequence by one position.

\begin{definition}
For $k\geq1$, define $A_n^{(k)}:=R_{n+1}^{(k)}$ for all $n\geq0$.  We also set $A_m^{(k)}=0$ for all $  m<0 $. In particular, $A_0^{(k)}=1$. 
\end{definition}

From Definition~\ref{def:RMF} of $R_n^{(k)}$, we obtain
\begin{equation}
\label{eq:A-recurrence}
A_n^{(k)}
=
\sum_{j=1}^{\,2+((n+1)\bmod k)}
A_{n-j}^{(k)},
\qquad n\geq1.
\end{equation}

Thus the shifted sequence has precisely the normalization required for a
tiling enumeration.

\subsection{Endpoint-periodic tilings}

Let
\[
L_k(n)
:=
2+((n+1)\bmod k),
\qquad n\geq1.
\]
The quantity $L_k(n)$ will be interpreted as the largest tile length
permitted for a tile whose right endpoint is position $n$.

\begin{definition}
\label{def:periodic-tiling}
Let $k\geq1$. An \emph{RM-tiling of length $n$ and period $k$} is a
tiling of a $1\times n$ board by integer length tiles such that every tile
of length $j$ ending at position $s$ satisfies
\begin{equation}
\label{eq:endpoint-condition}
1\leq j\leq L_k(s).
\end{equation}
The empty board has one tiling.
\end{definition}

The condition depends on the right endpoint of each tile, rather than on
its starting point. This convention is exactly what is required by the
recurrence: when the last tile of a tiling of length $n$ is removed, the
remaining object is an admissible tiling of length $n-j$ with no change
to any earlier endpoint condition.

Let $\mathcal{T}_n^{(k)}$ denote the set of RM-tilings of length $n$, and
write $T_n^{(k)}
=
\left|\mathcal{T}_n^{(k)}\right|$.  

\begin{theorem} 
\label{thm:tiling-interpretation}
For every $k\geq1$ and $n\geq0$, we have  $T_n^{(k)}
=
A_n^{(k)}
=
R_{n+1}^{(k)}$. 
\end{theorem}

\begin{proof}
First, we have $T_0^{(k)}=1=A_0^{(k)}$. Now, consider a nonempty RM-tiling of length $n$. If its final tile has length
$j$, then by Definition~\ref{def:periodic-tiling}, we have  $1\leq j\leq L_k(n)$. Removing this final tile leaves an arbitrary RM-tiling of length $n-j$.

Conversely, every RM-tiling of length $n-j$ can be extended uniquely by
appending a tile of length $j$, provided $1\leq j\leq L_k(n)$.  Hence the usual last tile decomposition gives
\[
T_n^{(k)}
=
\sum_{j=1}^{L_k(n)}T_{n-j}^{(k)}=\sum_{j=1}^{2+((n+1)\bmod k)}T_{n-j}^{(k)},
\]
which is exactly recurrence~\eqref{eq:A-recurrence}. The initial
conditions also agree, so $T_n^{(k)}=A_n^{(k)}=R_{n+1}^{(k)}$. 
\end{proof}

For $k=1$, one has $L_1(n)=2$ for every endpoint $n$. Thus Definition~\ref{def:periodic-tiling}
reduces to the familiar tiling of a board by monominoes and dominoes,
recovering the classical Fibonacci tiling interpretation
\cite{BenjaminQuinn2003,BenjaminHeberle2014}.

\subsection{Restricted compositions}

The same objects may be described   compactly as integer
compositions.

Recall that a composition of $n$ is an ordered tuple $\lambda=(\lambda_1,\ldots,\lambda_m)$ of positive integers satisfying $\lambda_1+\cdots+\lambda_m=n$.
 
For such a composition, define its partial sums $s_i
=
\lambda_1+\cdots+\lambda_i$ with  $1\leq i\leq m$. 
The number $s_i$ is the right endpoint of the tile corresponding to the
part $\lambda_i$.

\begin{definition}
A composition $\lambda=(\lambda_1,\ldots,\lambda_m)$ of $n$ is called a \emph{$k$-rotating composition} if for every $i$, we have 
\begin{equation}
\label{eq:composition-condition}
\lambda_i
\leq L_k(s_i).
\end{equation} 
\end{definition}

\begin{corollary}
\label{cor:composition-interpretation}
The number $R_{n+1}^{(k)}$ is equal to the number of
$k$-rotating compositions of $n$.
\end{corollary}

\begin{proof}
Read an RM-tiling from left to right and replace each tile by its length.
This produces a composition $(\lambda_1,\ldots,\lambda_m)$ of $n$. The endpoint of the $i$th tile is exactly $s_i$,  so the endpoint condition~\eqref{eq:endpoint-condition} becomes
\eqref{eq:composition-condition}. The construction is clearly reversible.
The result therefore follows from
Theorem~\ref{thm:tiling-interpretation}.
\end{proof}

This formulation highlights an unusual feature of the model. The
restriction on a part is determined not by the part number $i$ but by the
partial sum at which that part terminates. Thus the local rule is
periodic in the accumulated size of the composition.

Restricted compositions with periodic local conditions have been studied
in considerably broader settings; see, for example, Bender and Canfield
\cite{BenderCanfield2009}.  The restriction considered here is of a
different type.  It is not determined by the position of the part in the
composition, but by the partial sum at which that part terminates.
Thus the periodicity is attached to the accumulated size of the
composition rather than to its number of parts.

\subsection{A bijective explanation of doubling}

The algebraic identity $R_n^{(k)}=2R_{n-1}^{(k)}$ between reset positions of Lemma~\ref{lem:doubling} has a   simple bijective explanation.

In terms of the shifted sequence, the relevant condition is $n+1\not\equiv0\pmod{k}$. At such a position, $L_k(n)=L_k(n-1)+1$. 

\begin{theorem} 
\label{thm:binary-bijection}
Suppose $k\geq2$, $n\geq2$, and $n+1\not\equiv0\pmod{k}$.  Then there is a bijection $\mathcal{T}_n^{(k)}
\longleftrightarrow
\mathcal{T}_{n-1}^{(k)}\times\{0,1\}$.  Consequently, $T_n^{(k)}
=
2T_{n-1}^{(k)}$. 
\end{theorem}

\begin{proof}
Since $n+1\not\equiv0\pmod{k}$,  the allowed maximum tile length increases by one, i.e., $L_k(n)=L_k(n-1)+1$.

Partition $\mathcal{T}_n^{(k)}$ according to the length of the final tile.

\noindent
\emph{Case 0: the final tile has length $1$.} Delete this final tile. The remaining object is an arbitrary tiling in
$\mathcal{T}_{n-1}^{(k)}$. Conversely, every tiling of length $n-1$ can be extended by appending a
tile of length $1$. Hence this class is in bijection with
$\mathcal{T}_{n-1}^{(k)}$.

\noindent
\emph{Case 1: the final tile has length at least $2$.} Suppose the final tile has length $j\geq2$. Shorten that tile by one
unit. The resulting final tile has length $j-1$ and ends at position
$n-1$. Since $j\leq L_k(n)=L_k(n-1)+1$,  we have $j-1\leq L_k(n-1)$,  so the resulting tiling is valid. Conversely, take any tiling of length $n-1$ and increase the length of
its final tile by one. If its original length is $h$, then $h\leq L_k(n-1)$, and hence $h+1\leq L_k(n)$,  so the resulting tiling of length $n$ is valid. Thus this second class is also in bijection with
$\mathcal{T}_{n-1}^{(k)}$.

The two classes are disjoint and exhaustive, giving $\mathcal{T}_n^{(k)}
\longleftrightarrow
\mathcal{T}_{n-1}^{(k)}\times\{0,1\}$.  Therefore $T_n^{(k)}=2T_{n-1}^{(k)}$. 
\end{proof}

The two binary choices have a direct combinatorial meaning:
\[
\begin{array}{ccl}
0 &:& \text{append a new tile of length $1$},\\[1mm]
1 &:& \text{extend the final tile by one unit}.
\end{array}
\]
Thus each non-reset step contributes an independent binary decision.

\subsection{The ternary transition across a reset}

A complete memory cycle contains one exceptional transition. It is more
natural combinatorially to combine the reset with the immediately
preceding step.

Suppose $m+1\equiv0\pmod{k}$.  Then the endpoint $m$ is the last position before the memory resets,
while endpoint $m+1$ allows only tile lengths $1$ and $2$.

The recurrence gives $T_{m+1}^{(k)}
=
T_m^{(k)}+T_{m-1}^{(k)}$.  By Theorem~\ref{thm:binary-bijection}, $T_m^{(k)}=2T_{m-1}^{(k)}$.  Consequently, $T_{m+1}^{(k)}
=
3T_{m-1}^{(k)}$.

This identity also has a direct decomposition.

\begin{proposition} 
\label{prop:ternary-reset}
At every reset, there is a bijection
\[
\mathcal{T}_{m+1}^{(k)}
\longleftrightarrow
\mathcal{T}_{m-1}^{(k)}\times\{0,1,2\}.
\]
\end{proposition}

\begin{proof}
Every tiling in $\mathcal{T}_{m+1}^{(k)}$ ends with a tile of length
either $1$ or $2$.

If the last tile has length $2$, delete it. This gives one copy of
$\mathcal{T}_{m-1}^{(k)}$.

If the last tile has length $1$, delete it. The result lies in
$\mathcal{T}_m^{(k)}$. By
Theorem~\ref{thm:binary-bijection},
\[
\mathcal{T}_m^{(k)}
\longleftrightarrow
\mathcal{T}_{m-1}^{(k)}\times\{0,1\}.
\]
Hence the class of tilings ending in a tile of length $1$ gives two
additional copies of $\mathcal{T}_{m-1}^{(k)}$.

Altogether,
\[
\mathcal{T}_{m+1}^{(k)}
\longleftrightarrow
\mathcal{T}_{m-1}^{(k)}
\times
\{0,1,2\}.\qedhere
\]
\end{proof}

This provides a combinatorial explanation for the factor $3$ appearing
in the period multiplier.

\subsection{A combinatorial proof of the period-collapse theorem}

Combining the binary and ternary bijections yields a purely combinatorial
proof of the factor $3\cdot2^{k-2}$.  

\begin{theorem} 
\label{thm:period-code}
After the initial transient, extending an RM-tiling through one complete
period is equivalent to choosing one ternary symbol and $k-2$ binary
symbols. More precisely, for every $k\ge2$ and $n\ge k-1$, there is a bijection
\[
\mathcal{T}_{n+k}^{(k)}
\longleftrightarrow
\mathcal{T}_n^{(k)}
\times
\{0,1,2\}
\times
\{0,1\}^{k-2}.
\]
Consequently,
\[
T_{n+k}^{(k)}
=
3\cdot2^{k-2}T_n^{(k)}.
\]
\end{theorem}

\begin{proof}
During a complete period there are $k-2$ ordinary single step extensions
that are covered independently by
Theorem~\ref{thm:binary-bijection}. Each contributes a binary choice.

The remaining two consecutive positions consist of the final
pre-reset position and the reset itself. By
Proposition~\ref{prop:ternary-reset}, these two positions together
contribute one ternary choice.

Thus the set of possible extensions through a complete cycle is naturally
encoded by $\{0,1,2\}\times\{0,1\}^{k-2}$,  whose cardinality is $3\cdot2^{k-2}$.  Therefore $T_{n+k}^{(k)}
=
3\cdot2^{k-2}T_n^{(k)}$. 

Using $T_n^{(k)}=R_{n+1}^{(k)}$,  this recovers the period-collapse theorem combinatorially.
\end{proof}

Theorem~\ref{thm:period-code} explains why the multiplier
$3\cdot2^{k-2}$ is so simple. It is not merely an algebraic consequence
of cancellation in the recurrence: one complete rotating memory cycle is
combinatorially composed of one ternary decision and $k-2$ independent
binary decisions.

\section{The Period \texorpdfstring{$2$}{} Case and Monomer-dimer Tilings}
\label{sec:triangular-period2}

We now describe a natural geometric realization of the period-$2$
RM-Fibonacci sequence. Unlike the endpoint-periodic tiling model of
Section~\ref{sec:tilings}, the alternating recurrence is now forced by
the geometry of the underlying graph itself.

Let $G_n$ denote the triangular chain on $n$ vertices illustrated in
Figure~\ref{fig:triangular-chain-period2}. More precisely, $G_n$ has
vertex set $V(G_n)=\{v_1,v_2,\ldots,v_n\}$,  and edge set
\[
E(G_n)
=
\bigl\{\{v_i,v_{i+1}\}:1\leq i\leq n-1\bigr\}
\cup
\bigl\{\{v_{2j-1},v_{2j+1}\}:1\leq j\leq \lfloor (n-1)/2\rfloor\bigr\}.
\]
Thus $G_n$ is obtained by chaining triangles along a horizontal base.

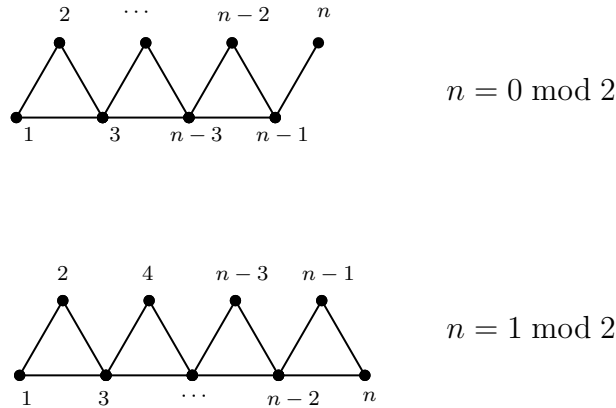
\begin{figure}[htbp]
\centering

\tikzset{every picture/.style={line width=0.75pt}} %set default line width to 0.75pt        

\begin{tikzpicture}[x=0.75pt,y=0.75pt,yscale=-1,xscale=1]
%uncomment if require: \path (0,421); %set diagram left start at 0, and has height of 421

%Straight Lines [id:da4867081249853029] 
\draw [color={rgb, 255:red, 0; green, 0; blue, 0 }  ,draw opacity=1 ]   (44.97,188.5) -- (66.62,151) ;
\draw [shift={(66.62,151)}, rotate = 300] [color={rgb, 255:red, 0; green, 0; blue, 0 }  ,draw opacity=1 ][fill={rgb, 255:red, 0; green, 0; blue, 0 }  ,fill opacity=1 ][line width=0.75]      (0, 0) circle [x radius= 2.34, y radius= 2.34]   ;
\draw [shift={(44.97,188.5)}, rotate = 300] [color={rgb, 255:red, 0; green, 0; blue, 0 }  ,draw opacity=1 ][fill={rgb, 255:red, 0; green, 0; blue, 0 }  ,fill opacity=1 ][line width=0.75]      (0, 0) circle [x radius= 2.34, y radius= 2.34]   ;
%Straight Lines [id:da8075956158726064] 
\draw [color={rgb, 255:red, 0; green, 0; blue, 0 }  ,draw opacity=1 ]   (1.67,188.5) -- (44.97,188.5) ;
\draw [shift={(44.97,188.5)}, rotate = 0] [color={rgb, 255:red, 0; green, 0; blue, 0 }  ,draw opacity=1 ][fill={rgb, 255:red, 0; green, 0; blue, 0 }  ,fill opacity=1 ][line width=0.75]      (0, 0) circle [x radius= 2.34, y radius= 2.34]   ;
\draw [shift={(1.67,188.5)}, rotate = 0] [color={rgb, 255:red, 0; green, 0; blue, 0 }  ,draw opacity=1 ][fill={rgb, 255:red, 0; green, 0; blue, 0 }  ,fill opacity=1 ][line width=0.75]      (0, 0) circle [x radius= 2.34, y radius= 2.34]   ;
%Straight Lines [id:da9996511774921507] 
\draw [color={rgb, 255:red, 0; green, 0; blue, 0 }  ,draw opacity=1 ]   (1.67,188.5) -- (23.32,151) ;
\draw [shift={(23.32,151)}, rotate = 300] [color={rgb, 255:red, 0; green, 0; blue, 0 }  ,draw opacity=1 ][fill={rgb, 255:red, 0; green, 0; blue, 0 }  ,fill opacity=1 ][line width=0.75]      (0, 0) circle [x radius= 2.34, y radius= 2.34]   ;
\draw [shift={(1.67,188.5)}, rotate = 300] [color={rgb, 255:red, 0; green, 0; blue, 0 }  ,draw opacity=1 ][fill={rgb, 255:red, 0; green, 0; blue, 0 }  ,fill opacity=1 ][line width=0.75]      (0, 0) circle [x radius= 2.34, y radius= 2.34]   ;
%Straight Lines [id:da11162251158886638] 
\draw [color={rgb, 255:red, 0; green, 0; blue, 0 }  ,draw opacity=1 ]   (88.27,188.5) -- (44.97,188.5) ;
\draw [shift={(44.97,188.5)}, rotate = 180] [color={rgb, 255:red, 0; green, 0; blue, 0 }  ,draw opacity=1 ][fill={rgb, 255:red, 0; green, 0; blue, 0 }  ,fill opacity=1 ][line width=0.75]      (0, 0) circle [x radius= 2.34, y radius= 2.34]   ;
\draw [shift={(88.27,188.5)}, rotate = 180] [color={rgb, 255:red, 0; green, 0; blue, 0 }  ,draw opacity=1 ][fill={rgb, 255:red, 0; green, 0; blue, 0 }  ,fill opacity=1 ][line width=0.75]      (0, 0) circle [x radius= 2.34, y radius= 2.34]   ;
%Straight Lines [id:da001099028079044495] 
\draw [color={rgb, 255:red, 0; green, 0; blue, 0 }  ,draw opacity=1 ]   (88.27,188.5) -- (66.62,151) ;
\draw [shift={(66.62,151)}, rotate = 240] [color={rgb, 255:red, 0; green, 0; blue, 0 }  ,draw opacity=1 ][fill={rgb, 255:red, 0; green, 0; blue, 0 }  ,fill opacity=1 ][line width=0.75]      (0, 0) circle [x radius= 2.34, y radius= 2.34]   ;
\draw [shift={(88.27,188.5)}, rotate = 240] [color={rgb, 255:red, 0; green, 0; blue, 0 }  ,draw opacity=1 ][fill={rgb, 255:red, 0; green, 0; blue, 0 }  ,fill opacity=1 ][line width=0.75]      (0, 0) circle [x radius= 2.34, y radius= 2.34]   ;
%Straight Lines [id:da39992960248122433] 
\draw [color={rgb, 255:red, 0; green, 0; blue, 0 }  ,draw opacity=1 ]   (131.57,188.5) -- (109.92,151) ;
\draw [shift={(109.92,151)}, rotate = 240] [color={rgb, 255:red, 0; green, 0; blue, 0 }  ,draw opacity=1 ][fill={rgb, 255:red, 0; green, 0; blue, 0 }  ,fill opacity=1 ][line width=0.75]      (0, 0) circle [x radius= 2.34, y radius= 2.34]   ;
\draw [shift={(131.57,188.5)}, rotate = 240] [color={rgb, 255:red, 0; green, 0; blue, 0 }  ,draw opacity=1 ][fill={rgb, 255:red, 0; green, 0; blue, 0 }  ,fill opacity=1 ][line width=0.75]      (0, 0) circle [x radius= 2.34, y radius= 2.34]   ;
%Straight Lines [id:da8319949522940872] 
\draw [color={rgb, 255:red, 0; green, 0; blue, 0 }  ,draw opacity=1 ]   (131.57,188.5) -- (88.27,188.5) ;
\draw [shift={(88.27,188.5)}, rotate = 180] [color={rgb, 255:red, 0; green, 0; blue, 0 }  ,draw opacity=1 ][fill={rgb, 255:red, 0; green, 0; blue, 0 }  ,fill opacity=1 ][line width=0.75]      (0, 0) circle [x radius= 2.34, y radius= 2.34]   ;
\draw [shift={(131.57,188.5)}, rotate = 180] [color={rgb, 255:red, 0; green, 0; blue, 0 }  ,draw opacity=1 ][fill={rgb, 255:red, 0; green, 0; blue, 0 }  ,fill opacity=1 ][line width=0.75]      (0, 0) circle [x radius= 2.34, y radius= 2.34]   ;
%Straight Lines [id:da28482038147869826] 
\draw [color={rgb, 255:red, 0; green, 0; blue, 0 }  ,draw opacity=1 ]   (174.87,188.5) -- (131.57,188.5) ;
\draw [shift={(131.57,188.5)}, rotate = 180] [color={rgb, 255:red, 0; green, 0; blue, 0 }  ,draw opacity=1 ][fill={rgb, 255:red, 0; green, 0; blue, 0 }  ,fill opacity=1 ][line width=0.75]      (0, 0) circle [x radius= 2.34, y radius= 2.34]   ;
\draw [shift={(174.87,188.5)}, rotate = 180] [color={rgb, 255:red, 0; green, 0; blue, 0 }  ,draw opacity=1 ][fill={rgb, 255:red, 0; green, 0; blue, 0 }  ,fill opacity=1 ][line width=0.75]      (0, 0) circle [x radius= 2.34, y radius= 2.34]   ;
%Straight Lines [id:da9951680792710607] 
\draw [color={rgb, 255:red, 0; green, 0; blue, 0 }  ,draw opacity=1 ]   (131.57,188.5) -- (153.22,151) ;
\draw [shift={(153.22,151)}, rotate = 300] [color={rgb, 255:red, 0; green, 0; blue, 0 }  ,draw opacity=1 ][fill={rgb, 255:red, 0; green, 0; blue, 0 }  ,fill opacity=1 ][line width=0.75]      (0, 0) circle [x radius= 2.34, y radius= 2.34]   ;
\draw [shift={(131.57,188.5)}, rotate = 300] [color={rgb, 255:red, 0; green, 0; blue, 0 }  ,draw opacity=1 ][fill={rgb, 255:red, 0; green, 0; blue, 0 }  ,fill opacity=1 ][line width=0.75]      (0, 0) circle [x radius= 2.34, y radius= 2.34]   ;
%Straight Lines [id:da1735789435175462] 
\draw [color={rgb, 255:red, 0; green, 0; blue, 0 }  ,draw opacity=1 ]   (44.97,188.5) -- (23.32,151) ;
\draw [shift={(23.32,151)}, rotate = 240] [color={rgb, 255:red, 0; green, 0; blue, 0 }  ,draw opacity=1 ][fill={rgb, 255:red, 0; green, 0; blue, 0 }  ,fill opacity=1 ][line width=0.75]      (0, 0) circle [x radius= 2.34, y radius= 2.34]   ;
\draw [shift={(44.97,188.5)}, rotate = 240] [color={rgb, 255:red, 0; green, 0; blue, 0 }  ,draw opacity=1 ][fill={rgb, 255:red, 0; green, 0; blue, 0 }  ,fill opacity=1 ][line width=0.75]      (0, 0) circle [x radius= 2.34, y radius= 2.34]   ;
%Straight Lines [id:da2751049721330667] 
\draw [color={rgb, 255:red, 0; green, 0; blue, 0 }  ,draw opacity=1 ]   (174.87,188.5) -- (153.22,151) ;
\draw [shift={(153.22,151)}, rotate = 240] [color={rgb, 255:red, 0; green, 0; blue, 0 }  ,draw opacity=1 ][fill={rgb, 255:red, 0; green, 0; blue, 0 }  ,fill opacity=1 ][line width=0.75]      (0, 0) circle [x radius= 2.34, y radius= 2.34]   ;
\draw [shift={(174.87,188.5)}, rotate = 240] [color={rgb, 255:red, 0; green, 0; blue, 0 }  ,draw opacity=1 ][fill={rgb, 255:red, 0; green, 0; blue, 0 }  ,fill opacity=1 ][line width=0.75]      (0, 0) circle [x radius= 2.34, y radius= 2.34]   ;
%Straight Lines [id:da8167347756473102] 
\draw [color={rgb, 255:red, 0; green, 0; blue, 0 }  ,draw opacity=1 ]   (88.27,188.5) -- (109.92,151) ;
\draw [shift={(109.92,151)}, rotate = 300] [color={rgb, 255:red, 0; green, 0; blue, 0 }  ,draw opacity=1 ][fill={rgb, 255:red, 0; green, 0; blue, 0 }  ,fill opacity=1 ][line width=0.75]      (0, 0) circle [x radius= 2.34, y radius= 2.34]   ;
\draw [shift={(88.27,188.5)}, rotate = 300] [color={rgb, 255:red, 0; green, 0; blue, 0 }  ,draw opacity=1 ][fill={rgb, 255:red, 0; green, 0; blue, 0 }  ,fill opacity=1 ][line width=0.75]      (0, 0) circle [x radius= 2.34, y radius= 2.34]   ;

%Straight Lines [id:da7200623675302318] 
\draw [color={rgb, 255:red, 0; green, 0; blue, 0 }  ,draw opacity=1 ]   (43.3,59.18) -- (64.95,21.68) ;
\draw [shift={(64.95,21.68)}, rotate = 300] [color={rgb, 255:red, 0; green, 0; blue, 0 }  ,draw opacity=1 ][fill={rgb, 255:red, 0; green, 0; blue, 0 }  ,fill opacity=1 ][line width=0.75]      (0, 0) circle [x radius= 2.34, y radius= 2.34]   ;
\draw [shift={(43.3,59.18)}, rotate = 300] [color={rgb, 255:red, 0; green, 0; blue, 0 }  ,draw opacity=1 ][fill={rgb, 255:red, 0; green, 0; blue, 0 }  ,fill opacity=1 ][line width=0.75]      (0, 0) circle [x radius= 2.34, y radius= 2.34]   ;
%Straight Lines [id:da040371690168022] 
\draw [color={rgb, 255:red, 0; green, 0; blue, 0 }  ,draw opacity=1 ]   (0,59.18) -- (43.3,59.18) ;
\draw [shift={(43.3,59.18)}, rotate = 0] [color={rgb, 255:red, 0; green, 0; blue, 0 }  ,draw opacity=1 ][fill={rgb, 255:red, 0; green, 0; blue, 0 }  ,fill opacity=1 ][line width=0.75]      (0, 0) circle [x radius= 2.34, y radius= 2.34]   ;
\draw [shift={(0,59.18)}, rotate = 0] [color={rgb, 255:red, 0; green, 0; blue, 0 }  ,draw opacity=1 ][fill={rgb, 255:red, 0; green, 0; blue, 0 }  ,fill opacity=1 ][line width=0.75]      (0, 0) circle [x radius= 2.34, y radius= 2.34]   ;
%Straight Lines [id:da2575236694115206] 
\draw [color={rgb, 255:red, 0; green, 0; blue, 0 }  ,draw opacity=1 ]   (0,59.18) -- (21.65,21.68) ;
\draw [shift={(21.65,21.68)}, rotate = 300] [color={rgb, 255:red, 0; green, 0; blue, 0 }  ,draw opacity=1 ][fill={rgb, 255:red, 0; green, 0; blue, 0 }  ,fill opacity=1 ][line width=0.75]      (0, 0) circle [x radius= 2.34, y radius= 2.34]   ;
\draw [shift={(0,59.18)}, rotate = 300] [color={rgb, 255:red, 0; green, 0; blue, 0 }  ,draw opacity=1 ][fill={rgb, 255:red, 0; green, 0; blue, 0 }  ,fill opacity=1 ][line width=0.75]      (0, 0) circle [x radius= 2.34, y radius= 2.34]   ;
%Straight Lines [id:da871108197320448] 
\draw [color={rgb, 255:red, 0; green, 0; blue, 0 }  ,draw opacity=1 ]   (86.6,59.18) -- (43.3,59.18) ;
\draw [shift={(43.3,59.18)}, rotate = 180] [color={rgb, 255:red, 0; green, 0; blue, 0 }  ,draw opacity=1 ][fill={rgb, 255:red, 0; green, 0; blue, 0 }  ,fill opacity=1 ][line width=0.75]      (0, 0) circle [x radius= 2.34, y radius= 2.34]   ;
\draw [shift={(86.6,59.18)}, rotate = 180] [color={rgb, 255:red, 0; green, 0; blue, 0 }  ,draw opacity=1 ][fill={rgb, 255:red, 0; green, 0; blue, 0 }  ,fill opacity=1 ][line width=0.75]      (0, 0) circle [x radius= 2.34, y radius= 2.34]   ;
%Straight Lines [id:da7580118388071287] 
\draw [color={rgb, 255:red, 0; green, 0; blue, 0 }  ,draw opacity=1 ]   (86.6,59.18) -- (64.95,21.68) ;
\draw [shift={(64.95,21.68)}, rotate = 240] [color={rgb, 255:red, 0; green, 0; blue, 0 }  ,draw opacity=1 ][fill={rgb, 255:red, 0; green, 0; blue, 0 }  ,fill opacity=1 ][line width=0.75]      (0, 0) circle [x radius= 2.34, y radius= 2.34]   ;
\draw [shift={(86.6,59.18)}, rotate = 240] [color={rgb, 255:red, 0; green, 0; blue, 0 }  ,draw opacity=1 ][fill={rgb, 255:red, 0; green, 0; blue, 0 }  ,fill opacity=1 ][line width=0.75]      (0, 0) circle [x radius= 2.34, y radius= 2.34]   ;
%Straight Lines [id:da19311048143556975] 
\draw [color={rgb, 255:red, 0; green, 0; blue, 0 }  ,draw opacity=1 ]   (129.9,59.18) -- (108.25,21.68) ;
\draw [shift={(108.25,21.68)}, rotate = 240] [color={rgb, 255:red, 0; green, 0; blue, 0 }  ,draw opacity=1 ][fill={rgb, 255:red, 0; green, 0; blue, 0 }  ,fill opacity=1 ][line width=0.75]      (0, 0) circle [x radius= 2.34, y radius= 2.34]   ;
\draw [shift={(129.9,59.18)}, rotate = 240] [color={rgb, 255:red, 0; green, 0; blue, 0 }  ,draw opacity=1 ][fill={rgb, 255:red, 0; green, 0; blue, 0 }  ,fill opacity=1 ][line width=0.75]      (0, 0) circle [x radius= 2.34, y radius= 2.34]   ;
%Straight Lines [id:da8845023255899412] 
\draw [color={rgb, 255:red, 0; green, 0; blue, 0 }  ,draw opacity=1 ]   (129.9,59.18) -- (86.6,59.18) ;
\draw [shift={(86.6,59.18)}, rotate = 180] [color={rgb, 255:red, 0; green, 0; blue, 0 }  ,draw opacity=1 ][fill={rgb, 255:red, 0; green, 0; blue, 0 }  ,fill opacity=1 ][line width=0.75]      (0, 0) circle [x radius= 2.34, y radius= 2.34]   ;
\draw [shift={(129.9,59.18)}, rotate = 180] [color={rgb, 255:red, 0; green, 0; blue, 0 }  ,draw opacity=1 ][fill={rgb, 255:red, 0; green, 0; blue, 0 }  ,fill opacity=1 ][line width=0.75]      (0, 0) circle [x radius= 2.34, y radius= 2.34]   ;
%Straight Lines [id:da7768651322126675] 
\draw [color={rgb, 255:red, 0; green, 0; blue, 0 }  ,draw opacity=1 ]   (129.9,59.18) -- (151.55,21.68) ;
\draw [shift={(151.55,21.68)}, rotate = 300] [color={rgb, 255:red, 0; green, 0; blue, 0 }  ,draw opacity=1 ][fill={rgb, 255:red, 0; green, 0; blue, 0 }  ,fill opacity=1 ][line width=0.75]      (0, 0) circle [x radius= 2.34, y radius= 2.34]   ;
\draw [shift={(129.9,59.18)}, rotate = 300] [color={rgb, 255:red, 0; green, 0; blue, 0 }  ,draw opacity=1 ][fill={rgb, 255:red, 0; green, 0; blue, 0 }  ,fill opacity=1 ][line width=0.75]      (0, 0) circle [x radius= 2.34, y radius= 2.34]   ;
%Straight Lines [id:da09815390516983957] 
\draw [color={rgb, 255:red, 0; green, 0; blue, 0 }  ,draw opacity=1 ]   (43.3,59.18) -- (21.65,21.68) ;
\draw [shift={(21.65,21.68)}, rotate = 240] [color={rgb, 255:red, 0; green, 0; blue, 0 }  ,draw opacity=1 ][fill={rgb, 255:red, 0; green, 0; blue, 0 }  ,fill opacity=1 ][line width=0.75]      (0, 0) circle [x radius= 2.34, y radius= 2.34]   ;
\draw [shift={(43.3,59.18)}, rotate = 240] [color={rgb, 255:red, 0; green, 0; blue, 0 }  ,draw opacity=1 ][fill={rgb, 255:red, 0; green, 0; blue, 0 }  ,fill opacity=1 ][line width=0.75]      (0, 0) circle [x radius= 2.34, y radius= 2.34]   ;
%Straight Lines [id:da38425962669170477] 
\draw [color={rgb, 255:red, 0; green, 0; blue, 0 }  ,draw opacity=1 ]   (86.6,59.18) -- (108.25,21.68) ;
\draw [shift={(108.25,21.68)}, rotate = 300] [color={rgb, 255:red, 0; green, 0; blue, 0 }  ,draw opacity=1 ][fill={rgb, 255:red, 0; green, 0; blue, 0 }  ,fill opacity=1 ][line width=0.75]      (0, 0) circle [x radius= 2.34, y radius= 2.34]   ;
\draw [shift={(86.6,59.18)}, rotate = 300] [color={rgb, 255:red, 0; green, 0; blue, 0 }  ,draw opacity=1 ][fill={rgb, 255:red, 0; green, 0; blue, 0 }  ,fill opacity=1 ][line width=0.75]      (0, 0) circle [x radius= 2.34, y radius= 2.34]   ;

% Text Node
\draw (214.33,157.07) node [anchor=north west][inner sep=0.75pt]    {$n=1\bmod 2$};
% Text Node
\draw (214.33,37.73) node [anchor=north west][inner sep=0.75pt]    {$n=0\bmod 2$};
% Text Node
\draw (2,62.58) node [anchor=north west][inner sep=0.75pt]  [font=\scriptsize]  {$1$};
% Text Node
\draw (20.02,2.4) node [anchor=north west][inner sep=0.75pt]  [font=\scriptsize]  {$2$};
% Text Node
\draw (45.3,62.58) node [anchor=north west][inner sep=0.75pt]  [font=\scriptsize]  {$3$};
% Text Node
\draw (149.35,2.4) node [anchor=north west][inner sep=0.75pt]  [font=\scriptsize]  {$n$};
% Text Node
\draw (52.02,2.4) node [anchor=north west][inner sep=0.75pt]  [font=\scriptsize]  {$\cdots $};
% Text Node
\draw (118.9,62.58) node [anchor=north west][inner sep=0.75pt]  [font=\scriptsize]  {$n-1$};
% Text Node
\draw (99.35,2.4) node [anchor=north west][inner sep=0.75pt]  [font=\scriptsize]  {$n-2$};
% Text Node
\draw (75.6,62.58) node [anchor=north west][inner sep=0.75pt]  [font=\scriptsize]  {$n-3$};
% Text Node
\draw (172.35,195.22) node [anchor=north west][inner sep=0.75pt]  [font=\scriptsize]  {$n$};
% Text Node
\draw (141.9,132.4) node [anchor=north west][inner sep=0.75pt]  [font=\scriptsize]  {$n-1$};
% Text Node
\draw (124.35,195.22) node [anchor=north west][inner sep=0.75pt]  [font=\scriptsize]  {$n-2$};
% Text Node
\draw (98.6,132.4) node [anchor=north west][inner sep=0.75pt]  [font=\scriptsize]  {$n-3$};
% Text Node
\draw (81.02,195.22) node [anchor=north west][inner sep=0.75pt]  [font=\scriptsize]  {$\cdots $};
% Text Node
\draw (1,195.22) node [anchor=north west][inner sep=0.75pt]  [font=\scriptsize]  {$1$};
% Text Node
\draw (19.02,132.4) node [anchor=north west][inner sep=0.75pt]  [font=\scriptsize]  {$2$};
% Text Node
\draw (39.63,195.22) node [anchor=north west][inner sep=0.75pt]  [font=\scriptsize]  {$3$};
% Text Node
\draw (61.63,132.4) node [anchor=north west][inner sep=0.75pt]  [font=\scriptsize]  {$4$};

\end{tikzpicture}
\caption{A triangular-chain model for the period-$2$ RM-Fibonacci
sequence.}
\label{fig:triangular-chain-period2}
\end{figure}

A \emph{monomer-dimer tiling} of $G_n$ is a covering of its vertices by
pairwise disjoint monomers and dimers, where a dimer occupies an edge of
$G_n$. Equivalently, it is a matching of $G_n$, with unmatched vertices
viewed as monomers.

Let $a_n$ denote the number of monomer-dimer tilings of $G_n$ in which
the leftmost vertex $v_1$ is forced to be a monomer. The next proposition shows that $(a_n)$ satisfies exactly the period-$2$
RM-Fibonacci recurrence.

\begin{theorem} 
\label{thm:triangular-period2-recurrence}
The sequence $(a_n)_{n\geq1}$ satisfies the recurrence 
\[
a_1=1,\; a_2=1,\; a_3=2,\quad \forall n\geq4: a_n
=
\begin{cases}
a_{n-1}+a_{n-2},
& n\equiv0\pmod2,\\[1mm]
a_{n-1}+a_{n-2}+a_{n-3},
& n\equiv1\pmod2.
\end{cases} 
\] 
\end{theorem}

\begin{proof}
The initial values are immediate. Now let $n\geq4$.

Suppose first that $n$ is even. Then the last vertex $v_n$ is the top
vertex of the final triangle and has only one neighbor, namely
$v_{n-1}$; see Figure~\ref{fig:triangular-chain-period2}. Hence there are
exactly two possibilities:

\begin{enumerate}
\item $v_n$ is a monomer. Removing $v_n$ leaves a tiling of $G_{n-1}$.

\item $v_n$ forms a dimer with $v_{n-1}$. Removing $v_{n-1}$ and $v_n$
leaves a tiling of $G_{n-2}$.
\end{enumerate}
Therefore $a_n=a_{n-1}+a_{n-2}$ with $n\equiv0\pmod2$.

Now suppose that $n$ is odd. Then the last vertex $v_n$ is the rightmost
base vertex and has two possible dimer partners, namely $v_{n-1}$ and
$v_{n-2}$; see again Figure~\ref{fig:triangular-chain-period2}. Hence
there are three possibilities:

\begin{enumerate}
\item $v_n$ is a monomer. Removing $v_n$ leaves a tiling of $G_{n-1}$.

\item $v_n$ forms a dimer with $v_{n-1}$. Removing $v_{n-1}$ and $v_n$
leaves a tiling of $G_{n-2}$.

\item $v_n$ forms a dimer with $v_{n-2}$. Then the vertex $v_{n-1}$ is
forced to be a monomer, and removing the last three vertices
$v_{n-2},v_{n-1},v_n$ leaves a tiling of $G_{n-3}$.
\end{enumerate}

Consequently, $a_n=a_{n-1}+a_{n-2}+a_{n-3}$ with $(n\equiv1\pmod2)$.
\end{proof}

We can now identify $(a_n)$ with the period-$2$ RM-Fibonacci sequence.

\begin{corollary}
\label{cor:triangular-period2-RM}
For every $n\geq1$, we have $a_n=R_n^{(2)}$. 
\end{corollary}

\begin{proof}
By direct identification. 
\end{proof}

\begin{remark}
If the leftmost vertex is not constrained, then the number of all
monomer-dimer tilings of $G_n$ is equal to $2R_n^{(2)}$ for every
$n\geq2$. Indeed, the vertex $v_1$ may be either a monomer or part of the
dimer $\{v_1,v_2\}$, and these two possibilities reduce respectively to a
tiling of $G_n$ with $v_1$ forced to be a monomer and to such a tiling of
$G_{n-1}$.
\end{remark}

The sequence obtained here is already recorded in the OEIS. More
precisely, $R_n^{(2)}=\seqnum{A038754}(n-2)$ for $n\geq2$. The OEIS entry for \seqnum{A038754} records several algebraic,
number-theoretic, and walk interpretations of the sequence, but does not
include the monomer--dimer interpretation on the triangular chain
developed above. Thus the present construction provides an additional
combinatorial realization of \seqnum{A038754}.

\section{The period-\texorpdfstring{$3$}{} Case and Double-Hexagon Tilings}
\label{sec:hexagonal}

The period-$3$ case admits a second natural geometric realization.
Recall that $A_n^{(3)}=R_{n+1}^{(3)}$,  whose first terms are
\[
1,1,2,4,8,12,24,48,72,144,288,432,\ldots.
\]
This is OEIS sequence \seqnum{A354541}, which is recorded as the number
of tilings of a double-hexagon strip of $n$ hexagonal cells by single
and double hexagons. We now derive this interpretation directly from
the geometry of the strip.

Let $\mathcal{H}_n$ denote the set of such tilings and write $H_n:=|\mathcal{H}_n|$. A single tile covers one hexagonal cell, whereas a double tile covers
two cells sharing an edge. We set $H_0=1$ and $H_m=0$ for all $m<0$. 

Figure~\ref{fig:period3-hexagon-cases} illustrates the three possible
terminal geometries of the strip according to the residue of $n$
modulo $3$. These three configurations are responsible for the
periodic recurrence satisfied by $(H_n)$.

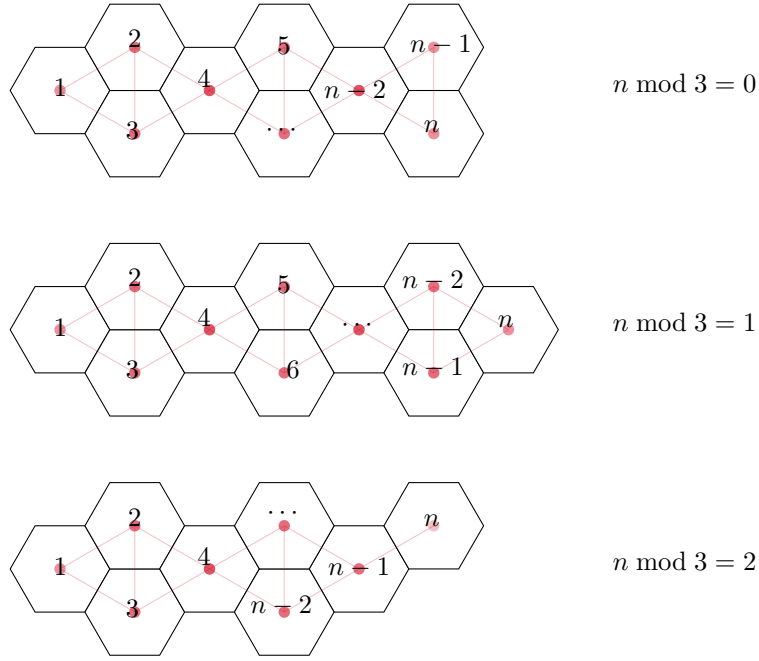
\begin{figure}[htbp]
\centering

% \tikzset{every picture/.style={line width=0.75pt}} %set default line width to 0.75pt        

\begin{tikzpicture}[x=0.75pt,y=0.75pt,yscale=-1,xscale=1]
%uncomment if require: \path (0,370); %set diagram left start at 0, and has height of 370

%Shape: Regular Polygon [id:dp9452661618642304] 
\draw  [fill={rgb, 255:red, 255; green, 255; blue, 255 }  ,fill opacity=1 ] (200,163.3) -- (187.5,184.95) -- (162.5,184.95) -- (150,163.3) -- (162.5,141.65) -- (187.5,141.65) -- cycle ;
%Shape: Regular Polygon [id:dp6283629860146436] 
\draw  [fill={rgb, 255:red, 255; green, 255; blue, 255 }  ,fill opacity=1 ] (50,163.3) -- (37.5,184.95) -- (12.5,184.95) -- (0,163.3) -- (12.5,141.65) -- (37.5,141.65) -- cycle ;
%Shape: Regular Polygon [id:dp9081432917484995] 
\draw  [fill={rgb, 255:red, 255; green, 255; blue, 255 }  ,fill opacity=1 ] (87.5,184.95) -- (75,206.6) -- (50,206.6) -- (37.5,184.95) -- (50,163.3) -- (75,163.3) -- cycle ;
%Shape: Regular Polygon [id:dp42806794043435203] 
\draw  [fill={rgb, 255:red, 255; green, 255; blue, 255 }  ,fill opacity=1 ] (87.5,141.65) -- (75,163.3) -- (50,163.3) -- (37.5,141.65) -- (50,120) -- (75,120) -- cycle ;
%Shape: Regular Polygon [id:dp43817383896840845] 
\draw  [fill={rgb, 255:red, 255; green, 255; blue, 255 }  ,fill opacity=1 ] (125,163.3) -- (112.5,184.95) -- (87.5,184.95) -- (75,163.3) -- (87.5,141.65) -- (112.5,141.65) -- cycle ;
%Shape: Regular Polygon [id:dp7359256598703617] 
\draw  [fill={rgb, 255:red, 255; green, 255; blue, 255 }  ,fill opacity=1 ] (162.5,184.95) -- (150,206.6) -- (125,206.6) -- (112.5,184.95) -- (125,163.3) -- (150,163.3) -- cycle ;
%Shape: Regular Polygon [id:dp10309429922432778] 
\draw  [fill={rgb, 255:red, 255; green, 255; blue, 255 }  ,fill opacity=1 ] (162.5,141.65) -- (150,163.3) -- (125,163.3) -- (112.5,141.65) -- (125,120) -- (150,120) -- cycle ;
%Shape: Regular Polygon [id:dp23718008900293874] 
\draw  [fill={rgb, 255:red, 255; green, 255; blue, 255 }  ,fill opacity=1 ] (237.5,141.65) -- (225,163.3) -- (200,163.3) -- (187.5,141.65) -- (200,120) -- (225,120) -- cycle ;
%Shape: Regular Polygon [id:dp9550249091458938] 
\draw  [fill={rgb, 255:red, 255; green, 255; blue, 255 }  ,fill opacity=1 ] (237.5,184.95) -- (225,206.6) -- (200,206.6) -- (187.5,184.95) -- (200,163.3) -- (225,163.3) -- cycle ;
%Shape: Regular Polygon [id:dp9176404080328591] 
\draw  [fill={rgb, 255:red, 255; green, 255; blue, 255 }  ,fill opacity=1 ] (275,163.3) -- (262.5,184.95) -- (237.5,184.95) -- (225,163.3) -- (237.5,141.65) -- (262.5,141.65) -- cycle ;

%Shape: Regular Polygon [id:dp6889708445063969] 
\draw  [fill={rgb, 255:red, 255; green, 255; blue, 255 }  ,fill opacity=1 ] (200,43.3) -- (187.5,64.95) -- (162.5,64.95) -- (150,43.3) -- (162.5,21.65) -- (187.5,21.65) -- cycle ;
%Shape: Regular Polygon [id:dp9021751969307163] 
\draw  [fill={rgb, 255:red, 255; green, 255; blue, 255 }  ,fill opacity=1 ] (50,43.3) -- (37.5,64.95) -- (12.5,64.95) -- (0,43.3) -- (12.5,21.65) -- (37.5,21.65) -- cycle ;
%Shape: Regular Polygon [id:dp5707350654645807] 
\draw  [fill={rgb, 255:red, 255; green, 255; blue, 255 }  ,fill opacity=1 ] (87.5,64.95) -- (75,86.6) -- (50,86.6) -- (37.5,64.95) -- (50,43.3) -- (75,43.3) -- cycle ;
%Shape: Regular Polygon [id:dp8285618917265746] 
\draw  [fill={rgb, 255:red, 255; green, 255; blue, 255 }  ,fill opacity=1 ] (87.5,21.65) -- (75,43.3) -- (50,43.3) -- (37.5,21.65) -- (50,0) -- (75,0) -- cycle ;
%Shape: Regular Polygon [id:dp9895659911564147] 
\draw  [fill={rgb, 255:red, 255; green, 255; blue, 255 }  ,fill opacity=1 ] (125,43.3) -- (112.5,64.95) -- (87.5,64.95) -- (75,43.3) -- (87.5,21.65) -- (112.5,21.65) -- cycle ;
%Shape: Regular Polygon [id:dp26192761906322726] 
\draw  [fill={rgb, 255:red, 255; green, 255; blue, 255 }  ,fill opacity=1 ] (162.5,64.95) -- (150,86.6) -- (125,86.6) -- (112.5,64.95) -- (125,43.3) -- (150,43.3) -- cycle ;
%Shape: Regular Polygon [id:dp45449173270438714] 
\draw  [fill={rgb, 255:red, 255; green, 255; blue, 255 }  ,fill opacity=1 ] (162.5,21.65) -- (150,43.3) -- (125,43.3) -- (112.5,21.65) -- (125,0) -- (150,0) -- cycle ;
%Shape: Regular Polygon [id:dp06307831688330456] 
\draw  [fill={rgb, 255:red, 255; green, 255; blue, 255 }  ,fill opacity=1 ] (237.5,21.65) -- (225,43.3) -- (200,43.3) -- (187.5,21.65) -- (200,0) -- (225,0) -- cycle ;
%Shape: Regular Polygon [id:dp4957656198370056] 
\draw  [fill={rgb, 255:red, 255; green, 255; blue, 255 }  ,fill opacity=1 ] (237.5,64.95) -- (225,86.6) -- (200,86.6) -- (187.5,64.95) -- (200,43.3) -- (225,43.3) -- cycle ;

%Shape: Regular Polygon [id:dp5944814154059312] 
\draw  [fill={rgb, 255:red, 255; green, 255; blue, 255 }  ,fill opacity=1 ] (200,283.3) -- (187.5,304.95) -- (162.5,304.95) -- (150,283.3) -- (162.5,261.65) -- (187.5,261.65) -- cycle ;
%Shape: Regular Polygon [id:dp10604891668686656] 
\draw  [fill={rgb, 255:red, 255; green, 255; blue, 255 }  ,fill opacity=1 ] (50,283.3) -- (37.5,304.95) -- (12.5,304.95) -- (0,283.3) -- (12.5,261.65) -- (37.5,261.65) -- cycle ;
%Shape: Regular Polygon [id:dp20542671685618097] 
\draw  [fill={rgb, 255:red, 255; green, 255; blue, 255 }  ,fill opacity=1 ] (87.5,304.95) -- (75,326.6) -- (50,326.6) -- (37.5,304.95) -- (50,283.3) -- (75,283.3) -- cycle ;
%Shape: Regular Polygon [id:dp32790961387459894] 
\draw  [fill={rgb, 255:red, 255; green, 255; blue, 255 }  ,fill opacity=1 ] (87.5,261.65) -- (75,283.3) -- (50,283.3) -- (37.5,261.65) -- (50,240) -- (75,240) -- cycle ;
%Shape: Regular Polygon [id:dp5787781536903209] 
\draw  [fill={rgb, 255:red, 255; green, 255; blue, 255 }  ,fill opacity=1 ] (125,283.3) -- (112.5,304.95) -- (87.5,304.95) -- (75,283.3) -- (87.5,261.65) -- (112.5,261.65) -- cycle ;
%Shape: Regular Polygon [id:dp03872509535848123] 
\draw  [fill={rgb, 255:red, 255; green, 255; blue, 255 }  ,fill opacity=1 ] (162.5,304.95) -- (150,326.6) -- (125,326.6) -- (112.5,304.95) -- (125,283.3) -- (150,283.3) -- cycle ;
%Shape: Regular Polygon [id:dp3930177608915617] 
\draw  [fill={rgb, 255:red, 255; green, 255; blue, 255 }  ,fill opacity=1 ] (162.5,261.65) -- (150,283.3) -- (125,283.3) -- (112.5,261.65) -- (125,240) -- (150,240) -- cycle ;
%Shape: Regular Polygon [id:dp015358793869362564] 
\draw  [fill={rgb, 255:red, 255; green, 255; blue, 255 }  ,fill opacity=1 ] (237.5,261.65) -- (225,283.3) -- (200,283.3) -- (187.5,261.65) -- (200,240) -- (225,240) -- cycle ;

%Straight Lines [id:da8503349012382275] 
\draw [color={rgb, 255:red, 208; green, 2; blue, 27 }  ,draw opacity=0.25 ]   (25,163.3) -- (62.5,141.65) ;
\draw [shift={(62.5,141.65)}, rotate = 330] [color={rgb, 255:red, 208; green, 2; blue, 27 }  ,draw opacity=0.25 ][fill={rgb, 255:red, 208; green, 2; blue, 27 }  ,fill opacity=0.25 ][line width=0.75]      (0, 0) circle [x radius= 2.34, y radius= 2.34]   ;
\draw [shift={(25,163.3)}, rotate = 330] [color={rgb, 255:red, 208; green, 2; blue, 27 }  ,draw opacity=0.25 ][fill={rgb, 255:red, 208; green, 2; blue, 27 }  ,fill opacity=0.25 ][line width=0.75]      (0, 0) circle [x radius= 2.34, y radius= 2.34]   ;
%Straight Lines [id:da0725134891120166] 
\draw [color={rgb, 255:red, 208; green, 2; blue, 27 }  ,draw opacity=0.25 ]   (62.5,184.95) -- (62.5,170.57) -- (62.5,141.65) ;
\draw [shift={(62.5,141.65)}, rotate = 270] [color={rgb, 255:red, 208; green, 2; blue, 27 }  ,draw opacity=0.25 ][fill={rgb, 255:red, 208; green, 2; blue, 27 }  ,fill opacity=0.25 ][line width=0.75]      (0, 0) circle [x radius= 2.34, y radius= 2.34]   ;
\draw [shift={(62.5,184.95)}, rotate = 270] [color={rgb, 255:red, 208; green, 2; blue, 27 }  ,draw opacity=0.25 ][fill={rgb, 255:red, 208; green, 2; blue, 27 }  ,fill opacity=0.25 ][line width=0.75]      (0, 0) circle [x radius= 2.34, y radius= 2.34]   ;
%Straight Lines [id:da9418474007424354] 
\draw [color={rgb, 255:red, 208; green, 2; blue, 27 }  ,draw opacity=0.25 ]   (62.5,184.95) -- (25,163.3) ;
\draw [shift={(25,163.3)}, rotate = 210] [color={rgb, 255:red, 208; green, 2; blue, 27 }  ,draw opacity=0.25 ][fill={rgb, 255:red, 208; green, 2; blue, 27 }  ,fill opacity=0.25 ][line width=0.75]      (0, 0) circle [x radius= 2.34, y radius= 2.34]   ;
\draw [shift={(62.5,184.95)}, rotate = 210] [color={rgb, 255:red, 208; green, 2; blue, 27 }  ,draw opacity=0.25 ][fill={rgb, 255:red, 208; green, 2; blue, 27 }  ,fill opacity=0.25 ][line width=0.75]      (0, 0) circle [x radius= 2.34, y radius= 2.34]   ;
%Straight Lines [id:da6656141514953978] 
\draw [color={rgb, 255:red, 208; green, 2; blue, 27 }  ,draw opacity=0.25 ]   (100,163.3) -- (62.5,141.65) ;
\draw [shift={(62.5,141.65)}, rotate = 210] [color={rgb, 255:red, 208; green, 2; blue, 27 }  ,draw opacity=0.25 ][fill={rgb, 255:red, 208; green, 2; blue, 27 }  ,fill opacity=0.25 ][line width=0.75]      (0, 0) circle [x radius= 2.34, y radius= 2.34]   ;
\draw [shift={(100,163.3)}, rotate = 210] [color={rgb, 255:red, 208; green, 2; blue, 27 }  ,draw opacity=0.25 ][fill={rgb, 255:red, 208; green, 2; blue, 27 }  ,fill opacity=0.25 ][line width=0.75]      (0, 0) circle [x radius= 2.34, y radius= 2.34]   ;
%Straight Lines [id:da36502135121383084] 
\draw [color={rgb, 255:red, 208; green, 2; blue, 27 }  ,draw opacity=0.25 ]   (175,163.3) -- (137.5,141.65) ;
\draw [shift={(137.5,141.65)}, rotate = 210] [color={rgb, 255:red, 208; green, 2; blue, 27 }  ,draw opacity=0.25 ][fill={rgb, 255:red, 208; green, 2; blue, 27 }  ,fill opacity=0.25 ][line width=0.75]      (0, 0) circle [x radius= 2.34, y radius= 2.34]   ;
\draw [shift={(175,163.3)}, rotate = 210] [color={rgb, 255:red, 208; green, 2; blue, 27 }  ,draw opacity=0.25 ][fill={rgb, 255:red, 208; green, 2; blue, 27 }  ,fill opacity=0.25 ][line width=0.75]      (0, 0) circle [x radius= 2.34, y radius= 2.34]   ;
%Straight Lines [id:da8801565563143302] 
\draw [color={rgb, 255:red, 208; green, 2; blue, 27 }  ,draw opacity=0.25 ]   (137.5,184.95) -- (100,163.3) ;
\draw [shift={(100,163.3)}, rotate = 210] [color={rgb, 255:red, 208; green, 2; blue, 27 }  ,draw opacity=0.25 ][fill={rgb, 255:red, 208; green, 2; blue, 27 }  ,fill opacity=0.25 ][line width=0.75]      (0, 0) circle [x radius= 2.34, y radius= 2.34]   ;
\draw [shift={(137.5,184.95)}, rotate = 210] [color={rgb, 255:red, 208; green, 2; blue, 27 }  ,draw opacity=0.25 ][fill={rgb, 255:red, 208; green, 2; blue, 27 }  ,fill opacity=0.25 ][line width=0.75]      (0, 0) circle [x radius= 2.34, y radius= 2.34]   ;
%Straight Lines [id:da746412701469881] 
\draw [color={rgb, 255:red, 208; green, 2; blue, 27 }  ,draw opacity=0.25 ]   (62.5,184.95) -- (100,163.3) ;
\draw [shift={(100,163.3)}, rotate = 330] [color={rgb, 255:red, 208; green, 2; blue, 27 }  ,draw opacity=0.25 ][fill={rgb, 255:red, 208; green, 2; blue, 27 }  ,fill opacity=0.25 ][line width=0.75]      (0, 0) circle [x radius= 2.34, y radius= 2.34]   ;
\draw [shift={(62.5,184.95)}, rotate = 330] [color={rgb, 255:red, 208; green, 2; blue, 27 }  ,draw opacity=0.25 ][fill={rgb, 255:red, 208; green, 2; blue, 27 }  ,fill opacity=0.25 ][line width=0.75]      (0, 0) circle [x radius= 2.34, y radius= 2.34]   ;
%Straight Lines [id:da5445448545788899] 
\draw [color={rgb, 255:red, 208; green, 2; blue, 27 }  ,draw opacity=0.25 ]   (100,163.3) -- (137.5,141.65) ;
\draw [shift={(137.5,141.65)}, rotate = 330] [color={rgb, 255:red, 208; green, 2; blue, 27 }  ,draw opacity=0.25 ][fill={rgb, 255:red, 208; green, 2; blue, 27 }  ,fill opacity=0.25 ][line width=0.75]      (0, 0) circle [x radius= 2.34, y radius= 2.34]   ;
\draw [shift={(100,163.3)}, rotate = 330] [color={rgb, 255:red, 208; green, 2; blue, 27 }  ,draw opacity=0.25 ][fill={rgb, 255:red, 208; green, 2; blue, 27 }  ,fill opacity=0.25 ][line width=0.75]      (0, 0) circle [x radius= 2.34, y radius= 2.34]   ;
%Straight Lines [id:da23244031601668635] 
\draw [color={rgb, 255:red, 208; green, 2; blue, 27 }  ,draw opacity=0.25 ]   (137.5,184.95) -- (175,163.3) ;
\draw [shift={(175,163.3)}, rotate = 330] [color={rgb, 255:red, 208; green, 2; blue, 27 }  ,draw opacity=0.25 ][fill={rgb, 255:red, 208; green, 2; blue, 27 }  ,fill opacity=0.25 ][line width=0.75]      (0, 0) circle [x radius= 2.34, y radius= 2.34]   ;
\draw [shift={(137.5,184.95)}, rotate = 330] [color={rgb, 255:red, 208; green, 2; blue, 27 }  ,draw opacity=0.25 ][fill={rgb, 255:red, 208; green, 2; blue, 27 }  ,fill opacity=0.25 ][line width=0.75]      (0, 0) circle [x radius= 2.34, y radius= 2.34]   ;
%Straight Lines [id:da977557322786548] 
\draw [color={rgb, 255:red, 208; green, 2; blue, 27 }  ,draw opacity=0.25 ]   (212.5,184.95) -- (250,163.3) ;
\draw [shift={(250,163.3)}, rotate = 330] [color={rgb, 255:red, 208; green, 2; blue, 27 }  ,draw opacity=0.25 ][fill={rgb, 255:red, 208; green, 2; blue, 27 }  ,fill opacity=0.25 ][line width=0.75]      (0, 0) circle [x radius= 2.34, y radius= 2.34]   ;
\draw [shift={(212.5,184.95)}, rotate = 330] [color={rgb, 255:red, 208; green, 2; blue, 27 }  ,draw opacity=0.25 ][fill={rgb, 255:red, 208; green, 2; blue, 27 }  ,fill opacity=0.25 ][line width=0.75]      (0, 0) circle [x radius= 2.34, y radius= 2.34]   ;
%Straight Lines [id:da6312233594737089] 
\draw [color={rgb, 255:red, 208; green, 2; blue, 27 }  ,draw opacity=0.25 ]   (175,163.3) -- (212.5,141.65) ;
\draw [shift={(212.5,141.65)}, rotate = 330] [color={rgb, 255:red, 208; green, 2; blue, 27 }  ,draw opacity=0.25 ][fill={rgb, 255:red, 208; green, 2; blue, 27 }  ,fill opacity=0.25 ][line width=0.75]      (0, 0) circle [x radius= 2.34, y radius= 2.34]   ;
\draw [shift={(175,163.3)}, rotate = 330] [color={rgb, 255:red, 208; green, 2; blue, 27 }  ,draw opacity=0.25 ][fill={rgb, 255:red, 208; green, 2; blue, 27 }  ,fill opacity=0.25 ][line width=0.75]      (0, 0) circle [x radius= 2.34, y radius= 2.34]   ;
%Straight Lines [id:da22283491480348394] 
\draw [color={rgb, 255:red, 208; green, 2; blue, 27 }  ,draw opacity=0.25 ]   (250,163.3) -- (212.5,141.65) ;
\draw [shift={(212.5,141.65)}, rotate = 210] [color={rgb, 255:red, 208; green, 2; blue, 27 }  ,draw opacity=0.25 ][fill={rgb, 255:red, 208; green, 2; blue, 27 }  ,fill opacity=0.25 ][line width=0.75]      (0, 0) circle [x radius= 2.34, y radius= 2.34]   ;
\draw [shift={(250,163.3)}, rotate = 210] [color={rgb, 255:red, 208; green, 2; blue, 27 }  ,draw opacity=0.25 ][fill={rgb, 255:red, 208; green, 2; blue, 27 }  ,fill opacity=0.25 ][line width=0.75]      (0, 0) circle [x radius= 2.34, y radius= 2.34]   ;
%Straight Lines [id:da9014774713759409] 
\draw [color={rgb, 255:red, 208; green, 2; blue, 27 }  ,draw opacity=0.25 ]   (212.5,184.95) -- (175,163.3) ;
\draw [shift={(175,163.3)}, rotate = 210] [color={rgb, 255:red, 208; green, 2; blue, 27 }  ,draw opacity=0.25 ][fill={rgb, 255:red, 208; green, 2; blue, 27 }  ,fill opacity=0.25 ][line width=0.75]      (0, 0) circle [x radius= 2.34, y radius= 2.34]   ;
\draw [shift={(212.5,184.95)}, rotate = 210] [color={rgb, 255:red, 208; green, 2; blue, 27 }  ,draw opacity=0.25 ][fill={rgb, 255:red, 208; green, 2; blue, 27 }  ,fill opacity=0.25 ][line width=0.75]      (0, 0) circle [x radius= 2.34, y radius= 2.34]   ;
%Straight Lines [id:da9315979840338656] 
\draw [color={rgb, 255:red, 208; green, 2; blue, 27 }  ,draw opacity=0.25 ]   (137.5,184.95) -- (137.5,141.65) ;
\draw [shift={(137.5,141.65)}, rotate = 270] [color={rgb, 255:red, 208; green, 2; blue, 27 }  ,draw opacity=0.25 ][fill={rgb, 255:red, 208; green, 2; blue, 27 }  ,fill opacity=0.25 ][line width=0.75]      (0, 0) circle [x radius= 2.34, y radius= 2.34]   ;
\draw [shift={(137.5,184.95)}, rotate = 270] [color={rgb, 255:red, 208; green, 2; blue, 27 }  ,draw opacity=0.25 ][fill={rgb, 255:red, 208; green, 2; blue, 27 }  ,fill opacity=0.25 ][line width=0.75]      (0, 0) circle [x radius= 2.34, y radius= 2.34]   ;
%Straight Lines [id:da6363462403439906] 
\draw [color={rgb, 255:red, 208; green, 2; blue, 27 }  ,draw opacity=0.25 ]   (212.5,184.95) -- (212.5,141.65) ;
\draw [shift={(212.5,141.65)}, rotate = 270] [color={rgb, 255:red, 208; green, 2; blue, 27 }  ,draw opacity=0.25 ][fill={rgb, 255:red, 208; green, 2; blue, 27 }  ,fill opacity=0.25 ][line width=0.75]      (0, 0) circle [x radius= 2.34, y radius= 2.34]   ;
\draw [shift={(212.5,184.95)}, rotate = 270] [color={rgb, 255:red, 208; green, 2; blue, 27 }  ,draw opacity=0.25 ][fill={rgb, 255:red, 208; green, 2; blue, 27 }  ,fill opacity=0.25 ][line width=0.75]      (0, 0) circle [x radius= 2.34, y radius= 2.34]   ;
%Straight Lines [id:da6519737223189372] 
\draw [color={rgb, 255:red, 208; green, 2; blue, 27 }  ,draw opacity=0.25 ]   (25,43.3) -- (62.5,21.65) ;
\draw [shift={(62.5,21.65)}, rotate = 330] [color={rgb, 255:red, 208; green, 2; blue, 27 }  ,draw opacity=0.25 ][fill={rgb, 255:red, 208; green, 2; blue, 27 }  ,fill opacity=0.25 ][line width=0.75]      (0, 0) circle [x radius= 2.34, y radius= 2.34]   ;
\draw [shift={(25,43.3)}, rotate = 330] [color={rgb, 255:red, 208; green, 2; blue, 27 }  ,draw opacity=0.25 ][fill={rgb, 255:red, 208; green, 2; blue, 27 }  ,fill opacity=0.25 ][line width=0.75]      (0, 0) circle [x radius= 2.34, y radius= 2.34]   ;
%Straight Lines [id:da5839494652237217] 
\draw [color={rgb, 255:red, 208; green, 2; blue, 27 }  ,draw opacity=0.25 ]   (62.5,64.95) -- (62.5,21.65) ;
\draw [shift={(62.5,21.65)}, rotate = 270] [color={rgb, 255:red, 208; green, 2; blue, 27 }  ,draw opacity=0.25 ][fill={rgb, 255:red, 208; green, 2; blue, 27 }  ,fill opacity=0.25 ][line width=0.75]      (0, 0) circle [x radius= 2.34, y radius= 2.34]   ;
\draw [shift={(62.5,64.95)}, rotate = 270] [color={rgb, 255:red, 208; green, 2; blue, 27 }  ,draw opacity=0.25 ][fill={rgb, 255:red, 208; green, 2; blue, 27 }  ,fill opacity=0.25 ][line width=0.75]      (0, 0) circle [x radius= 2.34, y radius= 2.34]   ;
%Straight Lines [id:da9453593645349916] 
\draw [color={rgb, 255:red, 208; green, 2; blue, 27 }  ,draw opacity=0.25 ]   (62.5,64.95) -- (25,43.3) ;
\draw [shift={(25,43.3)}, rotate = 210] [color={rgb, 255:red, 208; green, 2; blue, 27 }  ,draw opacity=0.25 ][fill={rgb, 255:red, 208; green, 2; blue, 27 }  ,fill opacity=0.25 ][line width=0.75]      (0, 0) circle [x radius= 2.34, y radius= 2.34]   ;
\draw [shift={(62.5,64.95)}, rotate = 210] [color={rgb, 255:red, 208; green, 2; blue, 27 }  ,draw opacity=0.25 ][fill={rgb, 255:red, 208; green, 2; blue, 27 }  ,fill opacity=0.25 ][line width=0.75]      (0, 0) circle [x radius= 2.34, y radius= 2.34]   ;
%Straight Lines [id:da4558346817969249] 
\draw [color={rgb, 255:red, 208; green, 2; blue, 27 }  ,draw opacity=0.25 ]   (100,43.3) -- (62.5,21.65) ;
\draw [shift={(62.5,21.65)}, rotate = 210] [color={rgb, 255:red, 208; green, 2; blue, 27 }  ,draw opacity=0.25 ][fill={rgb, 255:red, 208; green, 2; blue, 27 }  ,fill opacity=0.25 ][line width=0.75]      (0, 0) circle [x radius= 2.34, y radius= 2.34]   ;
\draw [shift={(100,43.3)}, rotate = 210] [color={rgb, 255:red, 208; green, 2; blue, 27 }  ,draw opacity=0.25 ][fill={rgb, 255:red, 208; green, 2; blue, 27 }  ,fill opacity=0.25 ][line width=0.75]      (0, 0) circle [x radius= 2.34, y radius= 2.34]   ;
%Straight Lines [id:da6269164484579078] 
\draw [color={rgb, 255:red, 208; green, 2; blue, 27 }  ,draw opacity=0.25 ]   (175,43.3) -- (137.5,21.65) ;
\draw [shift={(137.5,21.65)}, rotate = 210] [color={rgb, 255:red, 208; green, 2; blue, 27 }  ,draw opacity=0.25 ][fill={rgb, 255:red, 208; green, 2; blue, 27 }  ,fill opacity=0.25 ][line width=0.75]      (0, 0) circle [x radius= 2.34, y radius= 2.34]   ;
\draw [shift={(175,43.3)}, rotate = 210] [color={rgb, 255:red, 208; green, 2; blue, 27 }  ,draw opacity=0.25 ][fill={rgb, 255:red, 208; green, 2; blue, 27 }  ,fill opacity=0.25 ][line width=0.75]      (0, 0) circle [x radius= 2.34, y radius= 2.34]   ;
%Straight Lines [id:da712800387373282] 
\draw [color={rgb, 255:red, 208; green, 2; blue, 27 }  ,draw opacity=0.25 ]   (137.5,64.95) -- (100,43.3) ;
\draw [shift={(100,43.3)}, rotate = 210] [color={rgb, 255:red, 208; green, 2; blue, 27 }  ,draw opacity=0.25 ][fill={rgb, 255:red, 208; green, 2; blue, 27 }  ,fill opacity=0.25 ][line width=0.75]      (0, 0) circle [x radius= 2.34, y radius= 2.34]   ;
\draw [shift={(137.5,64.95)}, rotate = 210] [color={rgb, 255:red, 208; green, 2; blue, 27 }  ,draw opacity=0.25 ][fill={rgb, 255:red, 208; green, 2; blue, 27 }  ,fill opacity=0.25 ][line width=0.75]      (0, 0) circle [x radius= 2.34, y radius= 2.34]   ;
%Straight Lines [id:da7082247442254679] 
\draw [color={rgb, 255:red, 208; green, 2; blue, 27 }  ,draw opacity=0.25 ]   (62.5,64.95) -- (100,43.3) ;
\draw [shift={(100,43.3)}, rotate = 330] [color={rgb, 255:red, 208; green, 2; blue, 27 }  ,draw opacity=0.25 ][fill={rgb, 255:red, 208; green, 2; blue, 27 }  ,fill opacity=0.25 ][line width=0.75]      (0, 0) circle [x radius= 2.34, y radius= 2.34]   ;
\draw [shift={(62.5,64.95)}, rotate = 330] [color={rgb, 255:red, 208; green, 2; blue, 27 }  ,draw opacity=0.25 ][fill={rgb, 255:red, 208; green, 2; blue, 27 }  ,fill opacity=0.25 ][line width=0.75]      (0, 0) circle [x radius= 2.34, y radius= 2.34]   ;
%Straight Lines [id:da34046201735600723] 
\draw [color={rgb, 255:red, 208; green, 2; blue, 27 }  ,draw opacity=0.25 ]   (100,43.3) -- (137.5,21.65) ;
\draw [shift={(137.5,21.65)}, rotate = 330] [color={rgb, 255:red, 208; green, 2; blue, 27 }  ,draw opacity=0.25 ][fill={rgb, 255:red, 208; green, 2; blue, 27 }  ,fill opacity=0.25 ][line width=0.75]      (0, 0) circle [x radius= 2.34, y radius= 2.34]   ;
\draw [shift={(100,43.3)}, rotate = 330] [color={rgb, 255:red, 208; green, 2; blue, 27 }  ,draw opacity=0.25 ][fill={rgb, 255:red, 208; green, 2; blue, 27 }  ,fill opacity=0.25 ][line width=0.75]      (0, 0) circle [x radius= 2.34, y radius= 2.34]   ;
%Straight Lines [id:da3854836205953882] 
\draw [color={rgb, 255:red, 208; green, 2; blue, 27 }  ,draw opacity=0.25 ]   (137.5,64.95) -- (175,43.3) ;
\draw [shift={(175,43.3)}, rotate = 330] [color={rgb, 255:red, 208; green, 2; blue, 27 }  ,draw opacity=0.25 ][fill={rgb, 255:red, 208; green, 2; blue, 27 }  ,fill opacity=0.25 ][line width=0.75]      (0, 0) circle [x radius= 2.34, y radius= 2.34]   ;
\draw [shift={(137.5,64.95)}, rotate = 330] [color={rgb, 255:red, 208; green, 2; blue, 27 }  ,draw opacity=0.25 ][fill={rgb, 255:red, 208; green, 2; blue, 27 }  ,fill opacity=0.25 ][line width=0.75]      (0, 0) circle [x radius= 2.34, y radius= 2.34]   ;
%Straight Lines [id:da2964799142638742] 
\draw [color={rgb, 255:red, 208; green, 2; blue, 27 }  ,draw opacity=0.25 ]   (175,43.3) -- (212.5,21.65) ;
\draw [shift={(212.5,21.65)}, rotate = 330] [color={rgb, 255:red, 208; green, 2; blue, 27 }  ,draw opacity=0.25 ][fill={rgb, 255:red, 208; green, 2; blue, 27 }  ,fill opacity=0.25 ][line width=0.75]      (0, 0) circle [x radius= 2.34, y radius= 2.34]   ;
\draw [shift={(175,43.3)}, rotate = 330] [color={rgb, 255:red, 208; green, 2; blue, 27 }  ,draw opacity=0.25 ][fill={rgb, 255:red, 208; green, 2; blue, 27 }  ,fill opacity=0.25 ][line width=0.75]      (0, 0) circle [x radius= 2.34, y radius= 2.34]   ;
%Straight Lines [id:da5866887885257475] 
\draw [color={rgb, 255:red, 208; green, 2; blue, 27 }  ,draw opacity=0.25 ]   (212.5,64.95) -- (175,43.3) ;
\draw [shift={(175,43.3)}, rotate = 210] [color={rgb, 255:red, 208; green, 2; blue, 27 }  ,draw opacity=0.25 ][fill={rgb, 255:red, 208; green, 2; blue, 27 }  ,fill opacity=0.25 ][line width=0.75]      (0, 0) circle [x radius= 2.34, y radius= 2.34]   ;
\draw [shift={(212.5,64.95)}, rotate = 210] [color={rgb, 255:red, 208; green, 2; blue, 27 }  ,draw opacity=0.25 ][fill={rgb, 255:red, 208; green, 2; blue, 27 }  ,fill opacity=0.25 ][line width=0.75]      (0, 0) circle [x radius= 2.34, y radius= 2.34]   ;
%Straight Lines [id:da5075797720387306] 
\draw [color={rgb, 255:red, 208; green, 2; blue, 27 }  ,draw opacity=0.25 ]   (137.5,64.95) -- (137.5,21.65) ;
\draw [shift={(137.5,21.65)}, rotate = 270] [color={rgb, 255:red, 208; green, 2; blue, 27 }  ,draw opacity=0.25 ][fill={rgb, 255:red, 208; green, 2; blue, 27 }  ,fill opacity=0.25 ][line width=0.75]      (0, 0) circle [x radius= 2.34, y radius= 2.34]   ;
\draw [shift={(137.5,64.95)}, rotate = 270] [color={rgb, 255:red, 208; green, 2; blue, 27 }  ,draw opacity=0.25 ][fill={rgb, 255:red, 208; green, 2; blue, 27 }  ,fill opacity=0.25 ][line width=0.75]      (0, 0) circle [x radius= 2.34, y radius= 2.34]   ;
%Straight Lines [id:da6555575060566394] 
\draw [color={rgb, 255:red, 208; green, 2; blue, 27 }  ,draw opacity=0.25 ]   (212.5,64.95) -- (212.5,21.65) ;
\draw [shift={(212.5,21.65)}, rotate = 270] [color={rgb, 255:red, 208; green, 2; blue, 27 }  ,draw opacity=0.25 ][fill={rgb, 255:red, 208; green, 2; blue, 27 }  ,fill opacity=0.25 ][line width=0.75]      (0, 0) circle [x radius= 2.34, y radius= 2.34]   ;
\draw [shift={(212.5,64.95)}, rotate = 270] [color={rgb, 255:red, 208; green, 2; blue, 27 }  ,draw opacity=0.25 ][fill={rgb, 255:red, 208; green, 2; blue, 27 }  ,fill opacity=0.25 ][line width=0.75]      (0, 0) circle [x radius= 2.34, y radius= 2.34]   ;
%Straight Lines [id:da5483772680453416] 
\draw [color={rgb, 255:red, 208; green, 2; blue, 27 }  ,draw opacity=0.25 ]   (25,283.3) -- (62.5,261.65) ;
\draw [shift={(62.5,261.65)}, rotate = 330] [color={rgb, 255:red, 208; green, 2; blue, 27 }  ,draw opacity=0.25 ][fill={rgb, 255:red, 208; green, 2; blue, 27 }  ,fill opacity=0.25 ][line width=0.75]      (0, 0) circle [x radius= 2.34, y radius= 2.34]   ;
\draw [shift={(25,283.3)}, rotate = 330] [color={rgb, 255:red, 208; green, 2; blue, 27 }  ,draw opacity=0.25 ][fill={rgb, 255:red, 208; green, 2; blue, 27 }  ,fill opacity=0.25 ][line width=0.75]      (0, 0) circle [x radius= 2.34, y radius= 2.34]   ;
%Straight Lines [id:da4313313623926065] 
\draw [color={rgb, 255:red, 208; green, 2; blue, 27 }  ,draw opacity=0.25 ]   (62.5,304.95) -- (62.5,290.57) -- (62.5,261.65) ;
\draw [shift={(62.5,261.65)}, rotate = 270] [color={rgb, 255:red, 208; green, 2; blue, 27 }  ,draw opacity=0.25 ][fill={rgb, 255:red, 208; green, 2; blue, 27 }  ,fill opacity=0.25 ][line width=0.75]      (0, 0) circle [x radius= 2.34, y radius= 2.34]   ;
\draw [shift={(62.5,304.95)}, rotate = 270] [color={rgb, 255:red, 208; green, 2; blue, 27 }  ,draw opacity=0.25 ][fill={rgb, 255:red, 208; green, 2; blue, 27 }  ,fill opacity=0.25 ][line width=0.75]      (0, 0) circle [x radius= 2.34, y radius= 2.34]   ;
%Straight Lines [id:da32070015115560035] 
\draw [color={rgb, 255:red, 208; green, 2; blue, 27 }  ,draw opacity=0.25 ]   (62.5,304.95) -- (25,283.3) ;
\draw [shift={(25,283.3)}, rotate = 210] [color={rgb, 255:red, 208; green, 2; blue, 27 }  ,draw opacity=0.25 ][fill={rgb, 255:red, 208; green, 2; blue, 27 }  ,fill opacity=0.25 ][line width=0.75]      (0, 0) circle [x radius= 2.34, y radius= 2.34]   ;
\draw [shift={(62.5,304.95)}, rotate = 210] [color={rgb, 255:red, 208; green, 2; blue, 27 }  ,draw opacity=0.25 ][fill={rgb, 255:red, 208; green, 2; blue, 27 }  ,fill opacity=0.25 ][line width=0.75]      (0, 0) circle [x radius= 2.34, y radius= 2.34]   ;
%Straight Lines [id:da6719047431777027] 
\draw [color={rgb, 255:red, 208; green, 2; blue, 27 }  ,draw opacity=0.25 ]   (100,283.3) -- (62.5,261.65) ;
\draw [shift={(62.5,261.65)}, rotate = 210] [color={rgb, 255:red, 208; green, 2; blue, 27 }  ,draw opacity=0.25 ][fill={rgb, 255:red, 208; green, 2; blue, 27 }  ,fill opacity=0.25 ][line width=0.75]      (0, 0) circle [x radius= 2.34, y radius= 2.34]   ;
\draw [shift={(100,283.3)}, rotate = 210] [color={rgb, 255:red, 208; green, 2; blue, 27 }  ,draw opacity=0.25 ][fill={rgb, 255:red, 208; green, 2; blue, 27 }  ,fill opacity=0.25 ][line width=0.75]      (0, 0) circle [x radius= 2.34, y radius= 2.34]   ;
%Straight Lines [id:da21636225506917095] 
\draw [color={rgb, 255:red, 208; green, 2; blue, 27 }  ,draw opacity=0.25 ]   (175,283.3) -- (137.5,261.65) ;
\draw [shift={(137.5,261.65)}, rotate = 210] [color={rgb, 255:red, 208; green, 2; blue, 27 }  ,draw opacity=0.25 ][fill={rgb, 255:red, 208; green, 2; blue, 27 }  ,fill opacity=0.25 ][line width=0.75]      (0, 0) circle [x radius= 2.34, y radius= 2.34]   ;
\draw [shift={(175,283.3)}, rotate = 210] [color={rgb, 255:red, 208; green, 2; blue, 27 }  ,draw opacity=0.25 ][fill={rgb, 255:red, 208; green, 2; blue, 27 }  ,fill opacity=0.25 ][line width=0.75]      (0, 0) circle [x radius= 2.34, y radius= 2.34]   ;
%Straight Lines [id:da0492646834002477] 
\draw [color={rgb, 255:red, 208; green, 2; blue, 27 }  ,draw opacity=0.25 ]   (137.5,304.95) -- (100,283.3) ;
\draw [shift={(100,283.3)}, rotate = 210] [color={rgb, 255:red, 208; green, 2; blue, 27 }  ,draw opacity=0.25 ][fill={rgb, 255:red, 208; green, 2; blue, 27 }  ,fill opacity=0.25 ][line width=0.75]      (0, 0) circle [x radius= 2.34, y radius= 2.34]   ;
\draw [shift={(137.5,304.95)}, rotate = 210] [color={rgb, 255:red, 208; green, 2; blue, 27 }  ,draw opacity=0.25 ][fill={rgb, 255:red, 208; green, 2; blue, 27 }  ,fill opacity=0.25 ][line width=0.75]      (0, 0) circle [x radius= 2.34, y radius= 2.34]   ;
%Straight Lines [id:da9743557223333937] 
\draw [color={rgb, 255:red, 208; green, 2; blue, 27 }  ,draw opacity=0.25 ]   (62.5,304.95) -- (100,283.3) ;
\draw [shift={(100,283.3)}, rotate = 330] [color={rgb, 255:red, 208; green, 2; blue, 27 }  ,draw opacity=0.25 ][fill={rgb, 255:red, 208; green, 2; blue, 27 }  ,fill opacity=0.25 ][line width=0.75]      (0, 0) circle [x radius= 2.34, y radius= 2.34]   ;
\draw [shift={(62.5,304.95)}, rotate = 330] [color={rgb, 255:red, 208; green, 2; blue, 27 }  ,draw opacity=0.25 ][fill={rgb, 255:red, 208; green, 2; blue, 27 }  ,fill opacity=0.25 ][line width=0.75]      (0, 0) circle [x radius= 2.34, y radius= 2.34]   ;
%Straight Lines [id:da8078005780379889] 
\draw [color={rgb, 255:red, 208; green, 2; blue, 27 }  ,draw opacity=0.25 ]   (100,283.3) -- (137.5,261.65) ;
\draw [shift={(137.5,261.65)}, rotate = 330] [color={rgb, 255:red, 208; green, 2; blue, 27 }  ,draw opacity=0.25 ][fill={rgb, 255:red, 208; green, 2; blue, 27 }  ,fill opacity=0.25 ][line width=0.75]      (0, 0) circle [x radius= 2.34, y radius= 2.34]   ;
\draw [shift={(100,283.3)}, rotate = 330] [color={rgb, 255:red, 208; green, 2; blue, 27 }  ,draw opacity=0.25 ][fill={rgb, 255:red, 208; green, 2; blue, 27 }  ,fill opacity=0.25 ][line width=0.75]      (0, 0) circle [x radius= 2.34, y radius= 2.34]   ;
%Straight Lines [id:da04090246411668641] 
\draw [color={rgb, 255:red, 208; green, 2; blue, 27 }  ,draw opacity=0.25 ]   (137.5,304.95) -- (175,283.3) ;
\draw [shift={(175,283.3)}, rotate = 330] [color={rgb, 255:red, 208; green, 2; blue, 27 }  ,draw opacity=0.25 ][fill={rgb, 255:red, 208; green, 2; blue, 27 }  ,fill opacity=0.25 ][line width=0.75]      (0, 0) circle [x radius= 2.34, y radius= 2.34]   ;
\draw [shift={(137.5,304.95)}, rotate = 330] [color={rgb, 255:red, 208; green, 2; blue, 27 }  ,draw opacity=0.25 ][fill={rgb, 255:red, 208; green, 2; blue, 27 }  ,fill opacity=0.25 ][line width=0.75]      (0, 0) circle [x radius= 2.34, y radius= 2.34]   ;
%Straight Lines [id:da6250040381396814] 
\draw [color={rgb, 255:red, 208; green, 2; blue, 27 }  ,draw opacity=0.25 ]   (175,283.3) -- (212.5,261.65) ;
\draw [shift={(212.5,261.65)}, rotate = 330] [color={rgb, 255:red, 208; green, 2; blue, 27 }  ,draw opacity=0.25 ][fill={rgb, 255:red, 208; green, 2; blue, 27 }  ,fill opacity=0.25 ][line width=0.75]      (0, 0) circle [x radius= 2.34, y radius= 2.34]   ;
\draw [shift={(175,283.3)}, rotate = 330] [color={rgb, 255:red, 208; green, 2; blue, 27 }  ,draw opacity=0.25 ][fill={rgb, 255:red, 208; green, 2; blue, 27 }  ,fill opacity=0.25 ][line width=0.75]      (0, 0) circle [x radius= 2.34, y radius= 2.34]   ;
%Straight Lines [id:da0805941761428367] 
\draw [color={rgb, 255:red, 208; green, 2; blue, 27 }  ,draw opacity=0.25 ]   (137.5,304.95) -- (137.5,261.65) ;
\draw [shift={(137.5,261.65)}, rotate = 270] [color={rgb, 255:red, 208; green, 2; blue, 27 }  ,draw opacity=0.25 ][fill={rgb, 255:red, 208; green, 2; blue, 27 }  ,fill opacity=0.25 ][line width=0.75]      (0, 0) circle [x radius= 2.34, y radius= 2.34]   ;
\draw [shift={(137.5,304.95)}, rotate = 270] [color={rgb, 255:red, 208; green, 2; blue, 27 }  ,draw opacity=0.25 ][fill={rgb, 255:red, 208; green, 2; blue, 27 }  ,fill opacity=0.25 ][line width=0.75]      (0, 0) circle [x radius= 2.34, y radius= 2.34]   ;

% Text Node
\draw (301,153.4) node [anchor=north west][inner sep=0.75pt]    {\footnotesize$n\bmod 3=1$};
% Text Node
\draw (301,33.4) node [anchor=north west][inner sep=0.75pt]    {\footnotesize$n\bmod 3=0$};
% Text Node
\draw (301,273.4) node [anchor=north west][inner sep=0.75pt]    {\footnotesize$n\bmod 3=2$};
% Text Node
\draw (20.67,36.58) node [anchor=north west][inner sep=0.75pt]    {\footnotesize$1$};
% Text Node
\draw (58,11.91) node [anchor=north west][inner sep=0.75pt]    {\footnotesize$2$};
% Text Node
\draw (56.67,57.91) node [anchor=north west][inner sep=0.75pt]    {\footnotesize$3$};
% Text Node
\draw (92.67,31.91) node [anchor=north west][inner sep=0.75pt]    {\footnotesize$4$};
% Text Node
\draw (132.67,15.25) node [anchor=north west][inner sep=0.75pt]    {\footnotesize$5$};
% Text Node
\draw (127,59) node [anchor=north west][inner sep=0.75pt]    {\footnotesize$\cdots$};
% Text Node
\draw (206.33,57) node [anchor=north west][inner sep=0.75pt]    {\footnotesize$n$};
% Text Node
\draw (199.33,14.55) node [anchor=north west][inner sep=0.75pt]    {\footnotesize$n-1$};
% Text Node
\draw (156.33,37.55) node [anchor=north west][inner sep=0.75pt]    {\footnotesize$n-2$};
% Text Node
\draw (20.67,156.58) node [anchor=north west][inner sep=0.75pt]    {\footnotesize$1$};
% Text Node
\draw (58,131.91) node [anchor=north west][inner sep=0.75pt]    {\footnotesize$2$};
% Text Node
\draw (56.67,177.91) node [anchor=north west][inner sep=0.75pt]    {\footnotesize$3$};
% Text Node
\draw (92.67,151.91) node [anchor=north west][inner sep=0.75pt]    {\footnotesize$4$};
% Text Node
\draw (132.67,135.25) node [anchor=north west][inner sep=0.75pt]    {\footnotesize$5$};
% Text Node
\draw (137.33,177.91) node [anchor=north west][inner sep=0.75pt]    {\footnotesize$6$};
% Text Node
\draw (165,157) node [anchor=north west][inner sep=0.75pt]    {\footnotesize$\cdots$};
% Text Node
\draw (243.33,155) node [anchor=north west][inner sep=0.75pt]    {\footnotesize$n$};
% Text Node
\draw (195.33,176.55) node [anchor=north west][inner sep=0.75pt]    {\footnotesize$n-1$};
% Text Node
\draw (195.33,132.55) node [anchor=north west][inner sep=0.75pt]    {\footnotesize$n-2$};
% Text Node
\draw (20.67,276.58) node [anchor=north west][inner sep=0.75pt]    {\footnotesize$1$};
% Text Node
\draw (58,251.91) node [anchor=north west][inner sep=0.75pt]    {\footnotesize$2$};
% Text Node
\draw (56.67,297.91) node [anchor=north west][inner sep=0.75pt]    {\footnotesize$3$};
% Text Node
\draw (92.67,271.91) node [anchor=north west][inner sep=0.75pt]    {\footnotesize$4$};
% Text Node
\draw (127.67,251.25) node [anchor=north west][inner sep=0.75pt]    {\footnotesize$\cdots$};
% Text Node
\draw (206.33,255) node [anchor=north west][inner sep=0.75pt]    {\footnotesize$n$};
% Text Node
\draw (158.33,276.55) node [anchor=north west][inner sep=0.75pt]    {\footnotesize$n-1$};
% Text Node
\draw (119.33,295.55) node [anchor=north west][inner sep=0.75pt]    {\footnotesize$n-2$};

\end{tikzpicture}
\caption{Illustration of the period-$3$ hexagonal strip model. }
\label{fig:period3-hexagon-cases}
\end{figure}

\begin{theorem}
\label{thm:hexagon-recurrence}
For every $n\geq1$, we have 
\[
H_n=
\begin{cases}
H_{n-1}+H_{n-2}+H_{n-3},
& n\equiv0\pmod3,\\[2mm]
H_{n-1}+2H_{n-2},
& n\equiv1\pmod3,\\[2mm]
H_{n-1}+H_{n-2},
& n\equiv2\pmod3.
\end{cases}
\]
\end{theorem}

\begin{proof}
We distinguish the three terminal configurations shown in
Figure~\ref{fig:period3-hexagon-cases}.

Suppose first that $n\equiv0\pmod3$. The last cell $n$ may be covered by a single hexagon, in which case
the remaining strip is counted by $H_{n-1}$. If $n$ belongs to a
double hexagon, there are two possible terminal configurations. Pairing
$n$ with $n-1$ leaves a strip counted by $H_{n-2}$. Alternatively,
$n$ may be paired with $n-2$; in this case $n-1$ is forced to be a
single hexagon, and deleting the resulting terminal configuration
leaves a strip counted by $H_{n-3}$. Hence
\[
H_n=H_{n-1}+H_{n-2}+H_{n-3}.
\]

Now suppose that $n\equiv1\pmod3$. If cell $n$ is covered by a single hexagon, the remaining strip
contributes $H_{n-1}$. Otherwise, cell $n$ can be paired with either
of the two adjacent terminal cells shown in
Figure~\ref{fig:period3-hexagon-cases}. In either case, deleting the
double hexagon leaves a strip isomorphic to the strip on $n-2$ cells.
The two choices are distinct, and therefore
\[
H_n=H_{n-1}+2H_{n-2}.
\]

Finally, suppose that $n\equiv2\pmod3$. The last cell is either covered by a single hexagon, contributing
$H_{n-1}$, or belongs to the unique possible terminal double hexagon,
contributing $H_{n-2}$. Thus
\[
H_n=H_{n-1}+H_{n-2}.\qedhere
\]
\end{proof}

We now compare this geometric recurrence with the period-$3$
RM-Fibonacci recurrence. From~\eqref{eq:A-recurrence},
\begin{equation}
\label{eq:A3-rotating}
A_n^{(3)}
=
\begin{cases}
A_{n-1}^{(3)}
+A_{n-2}^{(3)}
+A_{n-3}^{(3)},
&
n\equiv0\pmod3,
\\[2mm]
A_{n-1}^{(3)}
+A_{n-2}^{(3)}
+A_{n-3}^{(3)}
+A_{n-4}^{(3)},
&
n\equiv1\pmod3,
\\[2mm]
A_{n-1}^{(3)}
+A_{n-2}^{(3)},
&
n\equiv2\pmod3.
\end{cases}
\end{equation}

Although the middle recurrence in
Theorem~\ref{thm:hexagon-recurrence} initially has a different
form, the two systems are equivalent. Indeed, if
$n\equiv1\pmod3$, then $n-2\equiv2\pmod3$, and hence
\[
H_{n-2}=H_{n-3}+H_{n-4}.
\]
Consequently,
\[
H_n
=
H_{n-1}+2H_{n-2}
=
H_{n-1}+H_{n-2}+H_{n-3}+H_{n-4}.
\]

Thus $(H_n)$ satisfies exactly the same rotating recurrence as
$(A_n^{(3)})$.

\begin{corollary}
\label{thm:hexagon-RM}
For every $n\geq0$,
\[
H_n=A_n^{(3)}=R_{n+1}^{(3)}.
\]
Consequently, the shifted period-$3$ RM-Fibonacci numbers enumerate
tilings of the double-hexagon strip by single and double hexagons and
coincide with OEIS sequence \seqnum{A354541}.
\end{corollary}

\begin{proof}
By direct identification. 
\end{proof}

The preceding argument proves the equinumeracy by a terminal
decomposition. In fact, the correspondence can be made explicit.

For convenience, write $[i]$ for a single hexagon occupying cell $i$,
and $[i,j]$ for a double hexagon covering the adjacent cells $i$ and
$j$.

\begin{theorem}
\label{thm:period3-bijection}
For every $n\geq0$, there is a bijection
\[
\Phi_n:
\mathcal{T}^{(3)}_n
\longrightarrow
\mathcal{H}_n.
\]
If an RM tile has length $j$ and right endpoint $s$, its image is given
by
\[
\begin{array}{c|c}
j & \text{hexagonal configuration}\\
\hline
1 & [s],\\[1mm]
2 & [s-1,s],\\[1mm]
3 & [s-1]\cup[s-2,s],\\[1mm]
4 & [s-3,s-1]\cup[s-2,s].
\end{array}
\]
\end{theorem}

\begin{proof}
For period $3$, the largest admissible RM tile length at endpoint $s$
is
\[
L_3(s)
=
\begin{cases}
3, & s\equiv0\pmod3,\\
4, & s\equiv1\pmod3,\\
2, & s\equiv2\pmod3.
\end{cases}
\]

With the labeling of Figure~\ref{fig:period3-hexagon-cases},
consecutive cells $s-1$ and $s$ are adjacent. Moreover, $s-2$ and $s$
are adjacent when $s\equiv0$ or $1\pmod3$, and, when
$s\equiv1\pmod3$, the cells $s-3$ and $s-1$ are also adjacent.
Consequently, every configuration in the table is well defined whenever
the corresponding RM tile length is admissible.

Since the RM tiles partition $\{1,\ldots,n\}$ into consecutive blocks,
applying the replacement to each block produces a tiling in
$\mathcal{H}_n$.

Conversely, consider the rightmost cell $s$ of a tiling in
$\mathcal{H}_n$. If $s$ is a single hexagon, it corresponds to an RM
tile of length $1$. If $s$ is paired with $s-1$, it corresponds to a
tile of length $2$. If $s$ is paired with $s-2$, then either $s-1$ is
single, giving a tile of length $3$, or, necessarily when
$s\equiv1\pmod3$, $s-1$ is paired with $s-3$, giving a tile of
length $4$.

These possibilities are mutually exclusive and exhaust the terminal
configurations shown in Figure~\ref{fig:period3-hexagon-cases}.
Removing the corresponding terminal block and repeating the procedure
from right to left recovers a unique RM tiling. Hence $\Phi_n$ is a
bijection.
\end{proof}

Thus the period-$3$ RM-Fibonacci sequence has two genuinely different
tiling interpretations:
\[
\mathcal{T}^{(3)}_n
\;\longleftrightarrow\;
\mathcal{H}_n,
\]
and Theorem~\ref{thm:period3-bijection} explains their equality
directly at the level of the underlying combinatorial objects.

\section{Concluding Remarks and Open Problems}
\label{sec:conclusion}

In this paper, we introduced the RM-Fibonacci numbers as a periodic
variable-order extension of the Fibonacci sequence and developed their
algebraic and combinatorial theory. Although the defining recurrence has
a periodically changing memory length, the resulting sequences exhibit a
remarkably rigid structure, including period collapse, explicit formulas,
rational generating functions, and exact exponential growth. On the
combinatorial side, RM-Fibonacci numbers arise from endpoint-periodic
tilings, restricted compositions, and directed lattice paths, while the
tiling model provides a bijective explanation for the multiplicative
factor accumulated over one complete memory cycle.

The first two nonclassical periods also admit natural geometric
realizations in which the periodicity is induced by the underlying
geometry rather than imposed explicitly. For period $2$, the sequence is
realized by monomer--dimer tilings of a triangular chain with a prescribed
boundary condition. For period $3$, the shifted sequence is realized by
tilings of a double-hexagon strip by single and double hexagons, and an
explicit bijection connects this model with the period-$3$ RM-tilings.
These examples suggest that geometry-driven realizations may exist more
generally.

\begin{problem}
For each $k\geq4$, construct, if possible, a natural periodic planar graph
or strip together with a fixed local tiling or matching rule whose
counting sequence is the period-$k$ RM-Fibonacci sequence. Ideally, the
cycle of recurrence orders $2,3,\ldots,k+1$ should arise from the geometry itself rather than from an explicitly
position-dependent rule.
\end{problem}

Another direction concerns the choice of the periodic memory pattern.
The present paper considers the consecutive cycle of memory lengths $2,3,\ldots,k+1$. More general periodic patterns may behave quite differently, and it is
natural to ask which of them retain the structural simplifications found
here.

\begin{problem}
Classify periodic memory patterns for which the associated variable-order
recurrence eventually reduces to a simple constant-coefficient
recurrence.
\end{problem}

The explicit arithmetic structure of the RM-Fibonacci numbers also raises
natural number-theoretic questions. Their divisibility properties,
greatest common divisors, periodicity modulo an integer, and related
congruence phenomena could be investigated systematically and compared
with their classical Fibonacci counterparts.

A further direction is to apply the rotating-memory principle to other
classical recurrence sequences. Periodic-memory analogues of Lucas, Pell,
Jacobsthal, Tribonacci, and related sequences may exhibit versions of the
same structural phenomena observed here.

\begin{problem}
Develop rotating-memory analogues of classical linear recurrence
sequences and determine which families retain properties such as period
collapse, geometric residue-class subsequences, rational generating
functions, and natural combinatorial interpretations.
\end{problem}

Finally, one may move beyond periodicity. Variable-order recurrences with
aperiodic, automatic, morphic, or random memory rules provide a broader
setting in which the rigid behavior found here is unlikely to persist
unchanged. Understanding the boundary between periodic systems with
strong algebraic structure and genuinely nonperiodic variable-memory
recurrences appears to be a natural direction for further study.

  \bibliographystyle{plainurl} 

 \bibliography{biblio}

 % {\color{red} ToDo: 

 % \begin{itemize}
 %     \item Update according to the new section 6.
 %     \item Ask for a generalization of the monomred-dimer case 2
 %     \item Add the simple worker example ? 
 %     \item mention that the case 2 is registred in the OEIS but this dimer monomer interpration is not mentioned
 % \end{itemize}
 % }

\end{document}